\documentclass[12pt]{amsart}
\usepackage{amsmath,amssymb,amsfonts,enumerate,amsthm}

\usepackage{graphicx}
\usepackage{subfigure}
\usepackage{color}

\usepackage[bookmarks,colorlinks,pagebackref]{hyperref} 

\headsep=1cm \numberwithin{equation}{section}
\newtheorem{thm}{Theorem}[section]

\newtheorem{cor}[thm]{Corollary}
\newtheorem{lem}[thm]{Lemma}
\newtheorem{lemma}[thm]{Lemma}
\newtheorem{lem defn}[thm]{Lemma and Definition}
\newtheorem{prop}[thm]{Proposition}

\newtheorem{defn rem}[thm]{Definition and Remark}
\newtheorem{exas}[thm]{Example}
\newtheorem{ind rem}[thm]{Introductory Remark}
\newtheorem{rem rem}[thm]{Reminder and Remark}
\newtheorem{exam}[thm]{Example}
\newtheorem{rem}[thm]{Remark}
\newtheorem{rem defn}[thm]{Remark and Definition}
\newtheorem{rem exam}[thm]{Remark and Example}

\newtheorem{nota rem}[thm]{Notation and Remark}
\newtheorem{prop defn}[thm]{Lemma and Definition}
\newtheorem{con not}[thm]{Convention and Notation}

\numberwithin{equation}{section}

\begin{document}

\bibliographystyle{amsplain}

\author[Brodmann]{Markus Brodmann}
\address{University of Z\"urich, Mathematics Institute, Winterthurerstrasse 190, CH -- 8057 Z\"urich. }
\email{brodmann@math.uzh.ch}

\author[Schenzel]{Peter Schenzel}
\address{Martin-Luther-Universit\"at Halle-Wittenberg,
Institut f\"ur Informatik, Von-Secken\-dorff-Platz 1, D -- 06120
Halle (Saale), Germany} \email{schenzel@informatik.uni-halle.de}

\date{Leipzig, Z\"urich, \today}


\subjclass[2000]{}

\title[FAMILIES OF BLOWUPS OF THE REAL AFFINE PLANE]
 {FAMILIES OF BLOWUPS OF THE REAL AFFINE PLANE: CLASSIFICATION, ISOTOPIES AND VISUALIZATIONS}

\begin{abstract} We classify embedded blowups of the real affine plane up to oriented isomorphy. We show that two blowups in the same isomorphism class are isotopic, using a matrix deformation argument similar to an idea given in \cite{Sh}. This answers two questions which were motivated by the interactive visualizations of such blowups (see \cite{SS}, \cite{St}, \cite{St2}).   
\end{abstract}
\maketitle

\section{Introduction}

\noindent {\rm \bf{The Visualization Project for Blowups of the Real Affine Plane.}} 
%One aim of this project is \textit{"to bring Algebraic Geometry to the Classroom"} already on early Undergraduate Level. 
{\rm
The present paper is primarily of theoretical nature. Nevertheless we begin with a "warm up" related to one aim of our whole Visualization Project, which is "to bring Algebraic Geometry to the Class Room" already on early undergraduate level. 
\begin{figure}
	\includegraphics[width=0.5\linewidth]{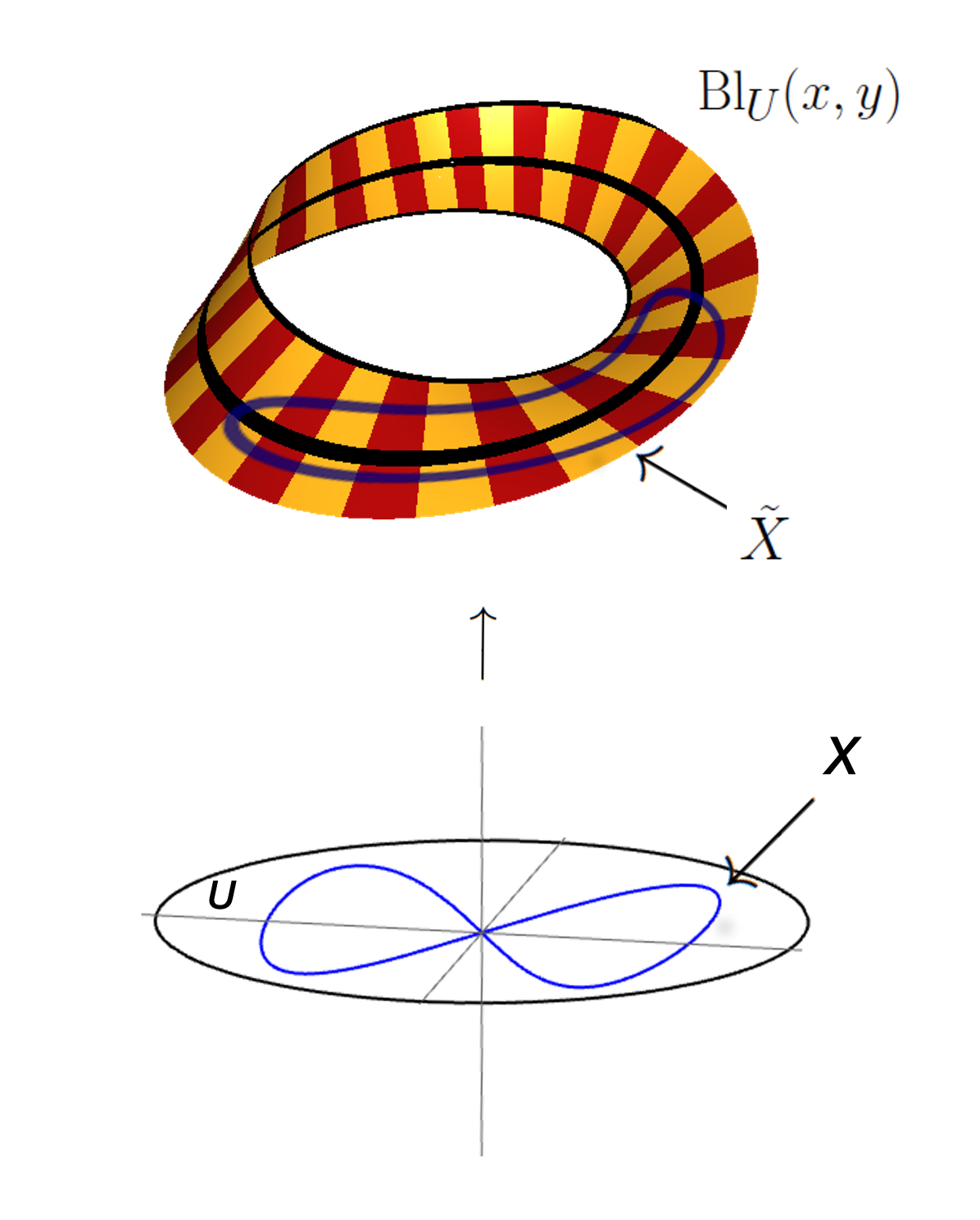}
	\caption[Blowing up the lemniscate]{The resolving effect of blowing up}
	\label{fig:moebius-2}
\end{figure}
Namely, in Figure 1 we illustrate the \textit{resolving effect of blowing up} -- a basic issue in Algebraic Geometry (see \cite{Hi} for the r\^ole of this effect  
	in the resolution of singularities of algebraic varieties in characteristic zero). Our example, which will be explained later in detail, shows how a simple nodal singularity of a plane curve is resolved by blowing up.}

{\rm Our paper is motivated by several investigations on the visualization of blowups of the real affine 
plane (see \cite{Bra},\cite{B1},\cite{B2},\cite{Kol},\cite{Kor},\cite{P}) in particular by the interactive 
visualizations suggested by the second named author and C. Stussak \cite{SS}.  
%(For a brief "historic" account on our Visualzation Project, see below.) 
Our principal aim is to consolidate the theoretical background of our Visualization Project and focuses on the following two problems:}  

\begin{itemize}
\item[\rm{(1.0)}] \begin{itemize}
\item[\rm{(a)}] \textit{Deformation Problem}: "Can one connect two arbitrary oriented isomorphic embedded 
blowups of the real affine plane by a continuous family within their isomorphism class?"
\item[{(b)}] \textit{Classification Problem}: "Is there a simple criterion to detect whether two regular 
embedded blowups of the affine plane are oriented isomorphic?"
\end{itemize}
\end{itemize}

We shall see, that both of these problems find an affirmative answer (see Theorem~\ref{Theorem REIS} and Theorem~\ref{Theorem EI1}). At first view, these are results of 
theoretical nature -- but, indeed, they also are of considerable practical meaning: Namely, once having tested that two 
embedded blowups $B$ and $C$ of the real affine plane are oriented isomorphic, one can use the animated 
visualization procedure of \cite{SS} to produce a family or sequence of pictures which shows a deformation between the two 
blowups $B$ and $C$ within their common isomorphism class. Moreover, our answer to the classification problem 
gives an easy way to detect whether two regular embedded blowups are oriented isomorphic. We shall provide a 
number of examples of this, including illustrations based on the visualization program {\sc RealSurf}  
developed by C. Stussak (see \cite{St}).\\

\noindent{\rm \bf{Blowups of the Real Affine Plane.}} We now start to set the precise setting in which we shall 
work. Let $Z \subset \mathbb{R}^2$ be a finite set and let $U \subset \mathbb{R}^2$ be an open bounded and 
star-shaped set with closure $\overline{U}$ such that $Z \subset U$ -- for example an open disk containing 
$Z.$ We fix a pair of two-variate real polynomials. 

\begin{itemize}
\item[\rm{(1.1)}] $\underline{f} :=  (f_0,f_1) \in \mathbb{R}[{\bf x}, {\bf y}]^2$ such that 
$\mathrm{Z}_{\overline{U}}(\underline{f}) := \{p \in \overline{U} \mid f_0(p) = f_1(p) = 0\} = Z. $
\end{itemize}

\noindent We always shall denote by  $\mathbb{P}^1 := \{(x_0:x_1) = [(x_0,x_1)] \mid (x_0,x_1) \in \mathbb{R}^2 \setminus\{(0,0)\}\}$
\noindent the \textit{real projective line}, whereas the complex projective line will be denoted by $\mathbb{P}^1_{\mathbb{C}}$. \\ 
For any set $S \subset U \times \mathbb{P}^1$ we denote by $\overline{S}$ 
the \textit{Zariski closure} of $S$ in $U \times \mathbb{P}^1,$ 
that is the restriction to $U \times \mathbb{P}^1$ of the closure of $S$ with respect to the 
Zariski topology in the ambient complex algebraic variety 
$\mathbb{A}^2_{\mathbb{C}} \times \mathbb{P}^1_{\mathbb{C}}$. 
Now, the \textit{embedded blowup} $\mathrm{Bl}_U(\underline{f})$ of $U$ with respect to the pair 
$\underline{f}$ is defined as the Zariski closure of the graph of the map

\begin{itemize}
\item[\rm{(1.2)}] $\varepsilon_{U,\underline{f}}:U \setminus Z \longrightarrow \mathbb{P}^1,$ given by $p \mapsto [\underline{f}(p)] = (f_0(p):f_1(p))$ 
\end{itemize}
\noindent in $U \times \mathbb{P}^1.$ More precisely, our embedded blowup is given by 

\begin{itemize}
\item[\rm{(1.3)}] \begin{itemize} \item[\rm{(a)}] the set $\mathrm{Bl}_U(\underline{f}) := 
\overline{\{\big(p,[\underline{f}(p)])\big) \mid p \in U \setminus Z\}}$ and
\item[\rm{(b)}] the \textit{canonical projection} map $\pi_{U,\underline{f}}:\mathrm{Bl}_U(\underline{f}) 
\longrightarrow U,$ given by $(p,(x_0:x_1)) \mapsto p$ for all $(p,(x_0:x_1)) \in \mathrm{Bl}_U(\underline{f}) 
\subset U \times \mathbb{P}^1.$ 
\end{itemize}
\end{itemize}

\begin{itemize}
\item[\rm{(1.4)}] \begin{itemize}
               \item[\rm{(a)}] {\rm The set $Z$ is called the \textit{center} of the blowup $\mathrm{Bl}_U(\underline{f}),$ whereas}
               \item[\rm{(b)}] {\rm the graph $\mathrm{Bl}^\circ_U(\underline{f}) := 
               \{(p,[\underline{f}(p)]) \mid p \in U \setminus Z\} = \mathrm{Bl}_U(\underline{f}) \setminus (Z \times \mathbb{P}^1)$ of $\varepsilon_{U,\underline{f}}$ 
               is called the \textit{open kernel} of our embedded blowup, and}
               \item[\rm{(c)}] {\rm the set $\mathrm{E}_U(\underline{f}) := \pi_{U,\underline{f}}^{-1}(Z) = 
               \mathrm{Bl}_U(\underline{f}) \cap (Z \times \mathbb{P}^1)$ 
               is called the \textit{exceptional set} of this embedded blowup.}  
               \end{itemize}
\end{itemize}               

{\rm Our basic aim is to study the class of embedded blowups}

\begin{itemize}
\item[\rm{(1.5)}] $\mathfrak{Bl}_U(Z) := \{\mathrm{Bl}_U(\underline{f}) \mid \underline{f} \in \mathbb{R}[{\bf x},{\bf y}]^2 \mbox{ {\rm with} } Z_{\overline{U}}(\underline{f}) = Z\},$
\end{itemize}

\noindent up to (relative oriented embedded) isomorphisms -- a concept which will be defined below. If we write  $\mathrm{Bl}_U(\underline{f}) \in \mathfrak{Bl}_U(Z)$ we tacitly mean that 
$(f_0, f_1) = \underline{f} \in \mathbb{R}[{\bf x},{\bf y}]^2$ satisfies the condition $Z_{\overline{U}}(\underline{f}) = Z.$  
\medskip

\medskip

\noindent{\rm \bf{Isomorphisms of Embedded Blowups.}} {\rm A \textit{(relative oriented) automorphism} (we often omit the wording 
in brackets from now on) of $U \times \mathbb{P}^1$ is a map }
 
\begin{itemize}
\item[\rm{(1.6)}] \begin{itemize}
\item[\rm{(a)}] $\varphi = \varphi_M:U \times \mathbb{P}^1 \longrightarrow U \times 
\mathbb{P}^1 \mbox{ {\rm given by} } (p,[\underline{v}]) \mapsto (p,[\underline{v}M(p)])$ 
for all $p \in U$ and all $\underline{v} \in \mathbb{R}^2 \setminus\{\underline{0}\},$ where
\item[\rm{(b)}] $M \in \mathbb{R}[{\bf x},{\bf y}]^{2 \times 2}$ with $\mathrm{det}\big(M(p)\big) >0$ 
for all $p \in U.$  
\end{itemize}
\end{itemize}
 
{\rm It is indeed justified to call these maps automorphisms. Namely: If $M \in 
\mathbb{R}[{\bf x},{\bf y}]^{2\times 2}$ with $\mathrm{det}\big(M(p)\big) >0$ for all $p \in U,$ its 
inverse $M^{-1} \in \mathbb{R}({\bf x},{\bf y})^{2 \times 2}$ may be written in the form 
$M^{-1} = \frac{1}{\mathrm{det}M}N$ with $N \in \mathbb{R}[{\bf x},{\bf y}]^{2 \times 2}$ and 
$\mathrm{det}\big(N(p)\big) = \mathrm{det}\big(M(p)\big) >0$ for all $p \in U.$ It is immediate, that the map $\varphi_N$ is inverse 
to $\varphi_M.$ Observe that a relative oriented automorphism of $U \times \mathbb{P}^1$ leaves fix 
the fiber  $\{p\} \times \mathbb{P}^1 \cong \mathbb{P}^1$ of the canonical projection 
$\pi:U \times \mathbb{P}^1 \longrightarrow U$ over each point $p \in U$ and acts as an orientation preserving
\textit{M\"obius-Transformation} on this fiber. 

We say that two embedded blowups $B, C \in \mathfrak{Bl}_U(Z)$ are \textit{(relatively oriented embedded) 
isomorphic} (we often omit the wording in brackets from now on) -- and write $B \cong C$ -- if there is an 
automorphism $\varphi$ of $U \times \mathbb{P}^1$ such that $C = \varphi(B).$ This means in particular:}
 
\begin{itemize}
\item[\rm{(1.7)}]  If $B = \mathrm{Bl}_U(\underline{f}), C \in \mathfrak{Bl}_U(Z),$ then $B \cong C$ if and only if 
$C = \mathrm{Bl}_U(\underline{f}M)$ for some $M \in \mathbb{R}[{\bf x},{\bf y}]^{2 \times 2}$ with 
$\mathrm{det}\big(M(p)\big) >0$ for all $p \in U.$ 
\end{itemize}

 \medskip

\noindent{\rm \bf{Regular Embedded Blowups and their Classification.}} {\rm We say that the pair $\underline{f} = (f_0,f_1) \in 
\mathbb{R}[{\bf x},{\bf y}]^2$ is \textit{regular} with respect to $Z$ on $U$ if:}
 
\begin{itemize}
\item[\rm{(1.8)}] \begin{itemize} 
\item[\rm{(a)}] $\mathrm{Z}_{\overline{U}}(\underline{f}) = Z.$ 
\item[\rm{(b)}] The Jacobian 
$$\partial \underline{f} := \begin{pmatrix}\frac{\partial f_0}{\partial {\bf x}} & 
\frac{\partial f_1}{\partial{\bf x}}\\{} {} \\ \frac{\partial f_0}{\partial {\bf y}} & 
\frac{\partial f_1}{\partial {\bf y}}\end{pmatrix} \in \mathbb{R}[{\bf x},{\bf y}]^{2 \times 2} \mbox{ {\rm of} }  
\underline{f} \mbox{ {\rm  is of rank} } 2 \mbox{ {\rm  in all points} } p \in Z.$$
\end{itemize}  
\end{itemize}
 
\noindent{\rm If the pair $\underline{f} \in \mathbb{R}[{\bf x},{\bf y}]^2$ is regular with respect to $Z$ on $U,$ 
we call $\mathrm{Bl}_U(\underline{f})$ a \textit{regular embedded blowup} of the set $U$ along $Z$ -- 
and we define:}
 
\begin{itemize}
\item[\rm{(1.9)}] $\mathfrak{Bl}^{\mathrm{reg}}_U(Z) := \{\mathrm{Bl}_U(\underline{f}) \mid 
\underline{f} \in \mathbb{R}[{\bf x},{\bf y}]^2 \mbox{ {\rm is regular with respect to} } Z 
\mbox{ {\rm on} }U\}.$  
\end{itemize}

\noindent{\rm If we write $\mathrm{Bl}_U(\underline{f}) \in \mathfrak{Bl}^{\mathrm{reg}}_U(Z),$ 
	we tacitly mean that 
$\underline{f} \in \mathbb{R}[{\bf x},{\bf y}]^2$ is regular with respect to $Z$ on $U.$} 
{\rm If $B = \mathrm{Bl}_U(\underline{f}) \in \mathfrak{Bl}^{\mathrm{reg}}_U(Z)$ 
 with $Z \not= \emptyset$, there is a map, depending only on $B,$} 
 
\begin{itemize}
\item[\rm{(1.10)}] $\mathrm{sgn}_B: Z \longrightarrow \{\pm 1 \}$ given by 
$p \mapsto \mathrm{sgn}\big(\mathrm{det}(\partial \underline{f}(p))\big)$ for all $p \in Z$ 
\end{itemize}
 
\noindent{\rm (see Definition and Remark~\ref{Definition sgn}), called the \textit{sign distribution} of $B.$ \\
{\rm If $B \in \mathfrak{Bl}^{\mathrm{reg}}_U(Z)$ with $\#Z = n \in \mathbb{N}$, we call $B$ a \textit{regular 
(embedded) $n$-point blowup}. We shall present examples of such $n$-point blowups and families of them for $n =1$ 
(see Example~\ref{Exam.Moebius}), for $n=2$ (see Examples~\ref{Example.Reg.3-Point} (B) and (C)), for $n=3$ 
(see Examples~\ref{Example.Reg.3-Point} (A) and (B)) and for $n=4$ (see Example~\ref{Exam.4Points}).}\\

\noindent{\rm Our \textit{Classification Problem} (1.0)(b) is 
answered as follows (see Theorem~\ref{Theorem EI1}):}
 
\begin{itemize} 
\item[\rm{(1.11)}] {\bf \textit{Classification Theorem}:} Two embedded blowups 
$B,C \in \mathfrak{Bl}^{\mathrm{reg}}_U(Z),$ are relatively oriented embedded isomorphic if and only 
if they have the same sign distribution. Hence, for short:  $B \cong C $ if and only if 
$\mathrm{sgn}_B = \mathrm{sgn}_C.$ 
\end{itemize}

\medskip 

\noindent{\rm \bf{Isotopies of Blowups and the Deformation Theorem.}} {\rm Now, we turn to the \textit{Deformation Problem} (1.0)(a). 
Given $B = \mathrm{Bl}_U(\underline{f}) \in \mathfrak{Bl}_Z(U),$ we are interested in families $\big(B^{(t)}\big)_{t \in [0,1]} \subset \mathfrak{Bl}_Z(U),$ such that $B^{(0)} = B$ and $B^{(t)} \cong B$ for all $t \in [0,1].$ In view of (1.6) and (1.7) 
it is natural to consider such families which come from an \textit{isotopy} of $U \times \mathbb{P}^1$-automorphisms. This means:}}
 
\begin{itemize}
\item[\rm{(1.12)}] There is a family of relative oriented $U \times \mathbb{P}^1$-automorphisms $\big(\varphi^{(t)} = \varphi_{M^{(t)}}\big)_{t \in [0,1]},$ given by a $(2 \times 2)$-matrix $\widetilde{M} \in \mathbb{R}[{\bf x},{\bf y},{\bf t}]^{2 \times 2},$ such that
                   \begin{itemize}
                   \item[\rm{(a)}] for all $t \in [0,1]$ and all $p \in U,$ the matrix $M^{(t)} := \widetilde{M}({\bf x},{\bf y},t) \in \mathbb{R}[{\bf x},{\bf y}]^{2 \times 2}$ satisfies $\mathrm{det}\big(M^{(t)}(p)\big) >0;$ 
                   \item[\rm{(b)}] $M^{(0)} = {\bf 1}^{2 \times 2}$ and $B^{(t)} = \varphi^{(t)}(B) = \mathrm{Bl}_U(\underline{f}M^{(t)})$ for all $t \in [0,1].$ 
                   \end{itemize}
\end{itemize}
 
{\rm In this context we shall solve the Deformation Problem (1.0)(a) (see Theorem~\ref{Theorem REIS}):}
 
\begin{itemize}
\item[\rm{(1.13)}] {\bf \textit{Deformation Theorem}:} Let $B, C \in \mathfrak{Bl}_U(Z)$  be 
relatively oriented embedded isomorphic. 
Then, $B$ and $C$ are connected by an isotopy of $U \times \mathbb{P}^1$-automorphisms. More precisely, 
there is an isotopy $\big(\varphi^{(t)} = \varphi_{M^{(t)}}\big)_{t \in [0,1]}$ as in (1.12) such that 
$\varphi^{(0)}(B) = B$ and $\varphi^{(1)}(B) = C.$   
\end{itemize}

\medskip 

\noindent{\rm \bf{Deformation of Matrices.}} {\rm Our Deformation Theorem (1.13) is a consequence of the following deformation result for matrices (see Proposition~\ref{Proposition RCF 1} and Remark~\ref{Remark RCF}):}
 
\begin{itemize}
\item[\rm{(1.14)}] {\bf \textit{Polynomial Deformations of Matrices}}: Let $M \in \mathbb{R}[{\bf x},{\bf y}]^{2 \times 2}$ such that $\mathrm{det}\big(M(p)\big) > 0$ for all $p \in \overline{U}.$ Then $M$ is connected to the unit matrix ${\bf 1}^{2 \times 2}  \in \mathbb{R}^{2 \times 2}$ by a polynomial family of $(2 \times 2)$-matrices with positive determinants on $\overline{U}$. More precisely:
There is a $(2 \times 2)$-matrix $\widetilde{M} \in \mathbb{R}[{\bf x},{\bf y},{\bf t}]^{2 \times 2},$ such that with $M^{(t)} := \widetilde{M}({\bf x},{\bf y},t) \quad$ 
we have
\begin{itemize} 
\item[\rm{(a)}] $\mathrm{det}\big(M^{(t)}(p)\big) > 0$ for all $t \in [0,1]$ and all $p \in \overline{U}.$
\item[\rm{(b)}] $M^{(0)} = {\bf 1}^{2 \times 2}$ and $M^{(1)} = M.$    
\end{itemize}
\end{itemize}  

\medskip

\noindent{\rm \bf{The Visualization Procedure.}} {\rm We now present a visualization procedure for embedded blowups  $B = \mathrm{Bl}_U(\underline{f}) \in \mathfrak{Bl}_U(Z).$ We use a method originally suggested in \cite{Bra} and
\cite{B1} -- in the modified form given in \cite{SS}. So, let $\rho, r \in \mathbb{R}$ with $0 < \rho < r$ and consider}
 
\begin{itemize}
\item[\rm{(1.15)}] \begin{itemize} 
\item[\rm{(a)}] the open disk $\mathbb{D} := \{(x,y) \in \mathbb{R}^2 \mid x^2 + y^2 < \rho^2\} \subset \mathbb{R}^2,$ with $U \subseteq \mathbb{D}$ and
\item[\rm{(b)}] the open solid torus $\mathbb{T} := \{(u,v,w) \in \mathbb{R}^3 \mid u^2 + \big(r-\sqrt{v^2+w^2}\big)^2 < \rho^2\} \subset \mathbb{R}^3$
\end{itemize}
\end{itemize}
 
{\rm together with the diffeomorphism}
 
\begin{itemize}
\item[\rm{(1.16)}] $\iota: \mathbb{D} \times \mathbb{P}^1 \stackrel{\cong}\longrightarrow \mathbb{T},$ given by 
$$\big((x,y),(x_0:x_1)\big) \mapsto \big(x,(r-y)\frac{x_0^2 - x_1^2}{x_0^2 + x_1^2}, (r-y)\frac{2x_0x_1}{x_0^2 + x_1^2}\big), \mbox{{\rm for all} } (x,y) \in U, (x_0:x_1) \in \mathbb{P}^1.$$
\end{itemize}
 
%{\rm We convene} 
 
\begin{itemize} 
\item[\rm{(1.17)}] The blowup $B = \mathrm{Bl}_U(\underline{f})$ is visualized by its diffeomorphic image 
$$\iota\big(\mathrm{Bl}_U(\underline{f})\big) = \iota\big(\mathrm{Bl}^\circ_U(\underline{f})\big)\; \dot{\cup} \; \iota\big(\mathrm{E}_U(\underline{f})\big) \subset \mathbb{T}, \mbox{ {\rm so that we have:} }$$ 
\begin{itemize}
\item[\rm{(a)}] $\iota\big(\mathrm{Bl}^\circ_U(\underline{f})\big) = \{\big(x,(r-y)\frac{f_0(x,y)^2 - f_1(x,y)^2}{f_0(x,y)^2 + f_1(x,y)^2}, (r-y)\frac{2f_0(x,y)f_1(x,y)}{f_0(x,y)^2 + f_1(x,y)^2}\mid (x,y) \in U \setminus Z\}.$
\item[\rm{(b)}] $\iota\big(\mathrm{E}_U(\underline{f})\big) \subseteq \iota(Z \times \mathbb{P}^1) = \dot{\bigcup}_{p \in Z} \; \iota(\{p\} \times \mathbb{P}^1).$ 
\item[\rm{(c)}] If $p=(x,y) \in Z$, then $\iota(\{p\} \times \mathbb{P}^1) \subset \mathbb{T}$ is the circle of radius $r-y$ given by:
\begin{align*} \iota(\{p\} \times \mathbb{P}^1) &= \{\big(x,(r-y)\frac{x_0^2 - x_1^2}{x_0^2 + x_1^2}, (r-y)\frac{2x_0x_1}{x_0^2 + x_1^2}\big) \mid (x_0,x_1) \in \mathbb{R}^2 \setminus \{(0,0)\}\} \\
 &= \{\big(x,(r-y)\mathrm{cos}(\beta), (r-y)\mathrm{sin}(\beta) \mid -\pi \leq \beta \leq \pi \}. 
\end{align*} 
\end{itemize}
\end{itemize}
 
{\rm Observe that $\iota\big(\mathrm{Bl}^\circ_U(\underline{f})\big) \subset \mathbb{T}$ is a surface without boundary and that $\iota(Z \times \mathbb{P}^1) \subset \mathbb{T}$ is a finite union 
of circles parallel to the central circle of $\mathbb{T}$ and centered at the rotation axis of $\mathbb{T}.$ \\ 
For each point $p \in Z,$ the accumulation points (with respect to the \textit{strict} (or \textit{metric}) \textit{topology}) of the open kernel 
$B^\circ = \mathrm{Bl}^\circ_U(\underline{f})$ in the fiber $\pi^{-1}_{U,\underline{f}}(p)$ are called 
the \textit{limit points} of $B$ above $p.$ We denote the set of these limit points by $\mathcal{L}_p(B),$ thus:}
\begin{itemize}
\item[\rm{(1.18)}] $\mathcal{L}_p(B) := \{(p,s) \in \{p\} \times \mathbb{P}^1 \mid \exists \big(p_n\big)_{n \in \mathbb{N}} \subset U \setminus Z: \mathrm{lim}_{n \rightarrow \infty} (p_n,\varepsilon_{U,\underline{f}}(p_n)) = (p,s)\}.$
\end{itemize}

The sets $\mathcal{L}_p(B)$ are of particular importance for the shape of the embedded blowup $B.$ Therefore, in some of our illustrations, their images  
$\iota\big(\mathcal{L}_p(B)\big)$ are colored in bold black and they usually appear as closed arcs on the circle $\iota(\big\{p\} \times \mathbb{P}^1\big).$ 

\medskip

\noindent{\rm \bf{The Technique of Visualization.}} For visualizations 
the parametric presentation given in (1.17) is used by Brandenberg (see \cite{Bra}) and also by 
Brodmann and Prager (see \cite{B1} and \cite{P}) for a very few examples. The difficulty 
of the parametrization for further examples is its instability in the neighborhood 
of $Z$ (see also Prager in \cite{P} for a further discussion). The new idea of C. Stussak 
(see \cite{St2} and \cite{SS}) -- which will be applied in this paper -- was to derive the implicit equation of the parametrized 
surface (based on the work of \cite{B1}) and to use the program {\sc{RealSurf}} (see \cite{St}) for its 
visualization. {\sc RealSurf} is a graphic GPU-program for the visualization of algebraic surfaces. It allows an interactive 
view of algebraic surfaces in $\mathbb{A}^3_{\mathbb{R}} = \mathbb{R}^3$ in real time. \\
%As already announced previously, all single pictures and sequences of pictures illustrating deformations of 
%blowups we present in this paper are build by the program {\sc RealSurf} developed by C. Stussak (see \cite{St}). 
In his PhD dissertation (see \cite{St2}) C. Stussak studied exact rasterization of algebraic curves and surfaces for the visualization on a personal computer with 
GPU-programming. As an application of his technique he and the second named author studied interactive visualizations of blowups of the real affine plane 
(see \cite{St2} and \cite{SS}). These interactive visualizations are based on {\sc RealSurf} with several adaptations for the particular situation of our concrete examples 
(see \cite{SS} for the technical details). The modified program allows continuous parameter changes by mouse action. With the help of these modifications we produced the 
pictures of the present paper. We are grateful to C. Stussak for making the adaption of {\sc{RealSurf}} 
available to us.

The pictures were produced on a PC with graphic cards {\sc nVidia} GT 525 {\sc Windows} 7.

\medskip

\noindent{\rm {\bf {A Few Preliminary Examples.}} {\rm Let us first recall the notion of \textit{affine standard charts} of an embedded blowup $B = \mathrm{Bl}_U(\underline{f}) \in \mathfrak{B}_U(Z),$
which are given by}
\begin{itemize}
 \item[\rm{(1.19)}] $\big(B\big)_i = \big(\mathrm{Bl}_U(\underline{f})\big)_i := \{(p,\frac{f_j(p)}{f_i(p)}) \in U \times \mathbb{R} \mid p \in U, f_i(p) \neq 0\} \quad (i,j \in \{0,1\}, i \neq j).$  
\end{itemize}
{\rm Keep in mind, that the blowup $B$ is obtained by pasting together the two affine standard charts $\big(B\big)_i, (i=0,1)$ by identifying (for $w \neq 0$) the two points $(p,w)$ and $(p,\frac{1}{w})$ of 
$U \times \mathbb{R}.$ Moreover, we can say: 
\begin{itemize}
 \item[\rm{(1.20)}] {\rm If $f_0$ and $f_1$ have no common divisor, then $\big(\mathrm{Bl}_U(\underline{f})\big)_i = Z_{U \times \mathbb{R}}(f_i{\bf w} - f_j), (i,j \in \{0,1\}, i\neq j\},$}
\end{itemize}
{\rm where $Z_{U\times \mathbb{R}}(h) := \{(x,y,w) \in U \times \mathbb{R} \mid h(x,y,z) = 0\}$ for $h \in \mathbb{R}[{\bf x},{\bf y},{\bf w}]. $\\}
\noindent{\rm To present two basic examples of blowups, we 
choose $\rho =2, r = 4, Z = \{(0,0)\}, U = \mathbb{D} = \{(x,y) \in \mathbb{R}^2 \mid x^2 + y^2 < 4\}.$ 
Then, for the choice $f_0 = {\bf x}, f_1 = {\bf y},$ the blowup $\mathrm{Bl}_U(\underline{f})$ is regular and appears as a \textit{M\"obius Strip} 
under our visualization process (see Figure 2 (a); see also \cite{H}, pg. 29, Figure 3, and \cite{Sha}, pg. 100, Figure 6, which both present 
sketches of an affine standard chart of this blowup). 

{\rm For the choice $f_0 = {\bf x}^2, f_1 = {\bf y}^2,$ the blowup $\mathrm{Bl}_U(\underline{f})$ is not regular and appears as 
a \textit{Double Whitney Umbrella} (see Figure 2 (b)). Indeed, according to (1.20) the two embedded standard affine charts of this blowup are given respectively 
by $Z_{U\times \mathbb{R}}({\bf x}^2{\bf w} - {\bf y}^2), Z_{U\times \mathbb{R}}({\bf y}^2{\bf w} - {\bf x}^2) 
\subset \mathbb{R}^3$ and hence appear as Whitney Umbrellas folded along the positive ${\bf w}$-axis and rotated 
around this axis with respect to each other by $90^\circ.$ 

 We now explain in detail the example shown in Figure 1 which illustrates the resolving effect of blowing up. 
	We choose $\rho, r, Z, U$ as above, set $f_0 = {\bf x},f_1 = {\bf y}$} and consider the lemniscate 
	$X :=  \{(x,y) \in \mathbb{R}^2 \mid x^2 - y^2 - \frac{1}{2}x^4 = 0\} = 
	Z_{\mathbb{A}^2_{\mathbb{C}}}({\bf x}^2 - {\bf y}^2 - \frac{1}{2}{\bf x}^4) \cap \mathbb{R}^2 \subset U,$ which has 
	a nodal singularity of multiplicity $2$ at the origin $\underline{0} := (0,0)$ and is smooth elsewhere.
	Finally we consider the so called \textit{strict transform}
$$
\widetilde{X} := \overline{\pi_{U,({\bf x},{\bf y})}^{-1}(X \setminus \{(0,0)\})} = 
\overline{\pi_{U, ({\bf x},{\bf y})}^{-1}(X) \cap \mathrm{Bl}_U^\circ({\bf x},{\bf y})} \subset 
\mathrm{Bl}_U({\bf x},{\bf y})
$$ 
{\rm of  $X,$ (with respect to the pair $({\bf x},{\bf y}) \in \mathbb{R}[{\bf x},{\bf y}]$) which is a non-singular curve contained in our embedded blowup $\mathrm{Bl}_U({\bf x},{\bf y})$ 
	(see Example 4.9.1 in Chapter I of \cite{H}) -- and hence appears as a smooth simple closed curve on a M\"obius strip -- as illustrated in Figure 2.
	The resolving effect of the same blowup is also illustrated on 
	an affine standard chart in  \cite{H}, pg. 29, Figure 3, but with a plane nodal quadric
	curve instead of a lemniscate. We did choose the lemniscate as its whole strict transform  
	appears on the blowup $\mathrm{Bl}_U({\bf x},{\bf y}).$}

\begin{figure}[t]
	\centering
	\begin{tabular}{cc}
		
		\includegraphics[width=.35\linewidth]{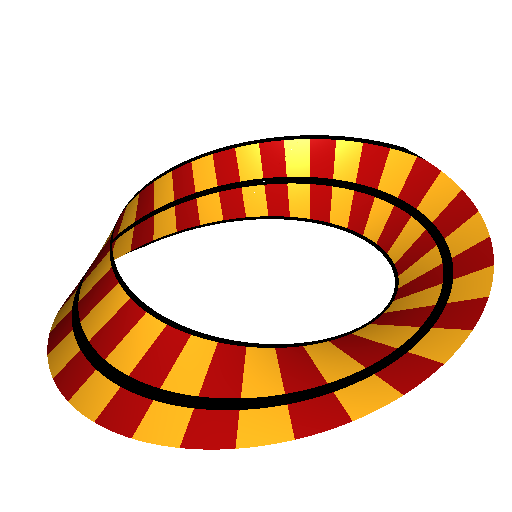}&  \\
		\small (a)
	\end{tabular} \qquad
	\begin{tabular}{c}
        \includegraphics[width=.4\linewidth]{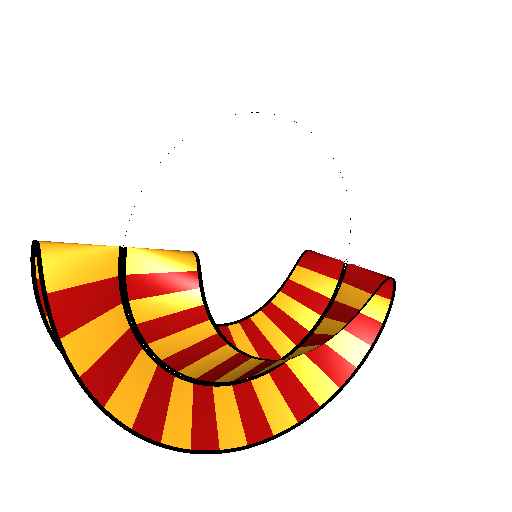} \\
		\small (b)
	\end{tabular}
	\caption{(a) \;M\"obius Strip $\quad\quad$ (b)\; Double Whitney Umbrella}
\end{figure}

\medskip

\noindent{\rm \bf{Acknowledgement.}} The authors thank the referees for their very careful and tedious study of the 
manuscript and their many critical comments. These lead us to perform a number of modifications and clarifications. They also thank the editor for his helpful hints concerning the final revision of the Manuscript. Finally they express their gratitude toward the \textit{Max-Planck-Institut f\"ur Mathematik in den Naturwissenschaften, Leipzig} for the offered hospitality during the preparation of this work.

\section{First Examples of Families of Blowups}

\noindent{\rm \bf{Examples and their Visualizations.}} {\rm We shall continue with a few examples of families of 
embedded blowups and their visualizations. Already now, we present three examples, which give a first flavor of 
the subject and illuminate some typical features. Again, as in the examples visualized by Figure 2, we choose 
$\rho =2, r = 4$ and $U = \mathbb{D} = \{(x,y) \in \mathbb{R}^2 \mid x^2 + y^2 < 4\}.$} 

\begin{exam}\label{Exam.Moebius} {\rm In our first example, we consider the most simple regular blowup of the real 
affine plane, namely the \textit{regular one-point blowup} 
$B: = \mathrm{Bl}_U({\bf x},{\bf y}),$ whose visualization shows up as a M\"obius strip (see Figure 2(a)). We 
deform this blowup by means of the family of polynomial matrices}
$$\big(M^{(t)} := \begin{pmatrix} 1-t & \frac{t}{2}\\ -\frac{t}{2} & 1+t \end{pmatrix}\big)_{t \in ]-\frac{2}{3}\sqrt{3},\frac{2}{3}\sqrt{3}[} 
\mbox{ {\rm with} } \mathrm{det}(M^{(t)}) = 1 - \frac{3}{4}t^2 >0 \mbox{ {\rm for} }  t 
\in ]-\frac{2}{3}\sqrt{3},\frac{2}{3}\sqrt{3}[.$$ 
{\rm  This leads us to the family of regular embedded blowups $\big(B^{(t)}\big)_{t \in ]-\frac{2}{3}\sqrt{3},\frac{2}{3}\sqrt{3}[}$ with} 
$$B^{(t)} := \mathrm{Bl}_U\big(({\bf x},{\bf y})M^{(t)}\big) = \mathrm{Bl}_U(f_0^{(t)}= 
(1-t){\bf x} - \frac{t}{2}{\bf y},f_1^{(t)} = \frac{t}{2}{\bf x} + (1+t){\bf y}\big) \in \mathfrak{Bl}^{\mathrm{reg}}_U\big(\{\underline{0}\}\big)$$
{\rm and }
$$Z = Z_U\big(\underline{f}^{(t)} := (f_0^{(t)},f_1^{(t)})\big) = \{(0,0)\} 
\mbox{ {\rm for all} } t \in ]-\frac{2}{3}\sqrt{3},\frac{2}{3}\sqrt{3}[.$$
{\rm In view of Figure 2(a) we expect that the visualization $\big(\iota(B^{(t))}\big)_{t \in ]-\frac{2}{3}\sqrt{3},\frac{2}{3}\sqrt{3}[}$ 
of this family presents itself as a deformation of a M\"obius strip. In Figure 3 we present this deformation 
for the values $t=0, t=0.4$ and $t=1.$ We also allow ourselves to leave the range $0 \leq t \leq 1$ and consider 
the three values $t=1.15, t=1.2$ and $t=1.4,$ which come close or lie beyond the critical value 
$t = \frac{2}{3}\sqrt{3} = 1.15470\ldots \quad.$ \\
These choices illustrate the following fact: If $t$ takes its critical values $\pm \frac{2}{3}\sqrt{3},$ 
the two linear forms $f_0^{(t)}$ and $f_1^{(t)}$ are linearly dependent and hence do not define a blowup in our 
sense. If $t \notin [-\frac{2}{3}\sqrt{3}, \frac{2}{3}\sqrt{3}]$ the blowup $B^{(t)}$ shows up again as a M\"obius 
strip, but reversely twisted along its central circle. }%but with reverse orientation.} 

\begin{figure}[b]
	\centering
	\begin{tabular}{ccc}
		\includegraphics[width=.3\linewidth]{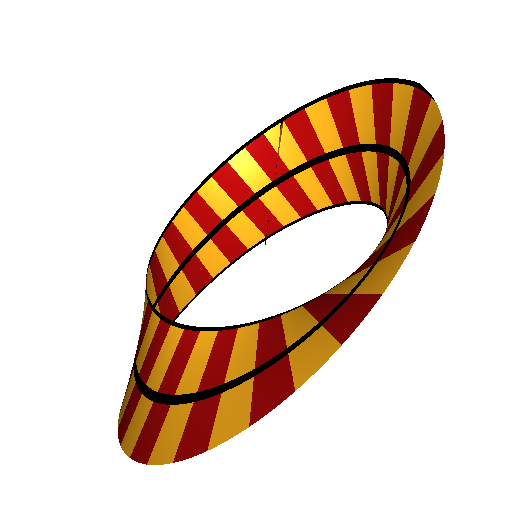}&  \includegraphics[width=.3\linewidth]{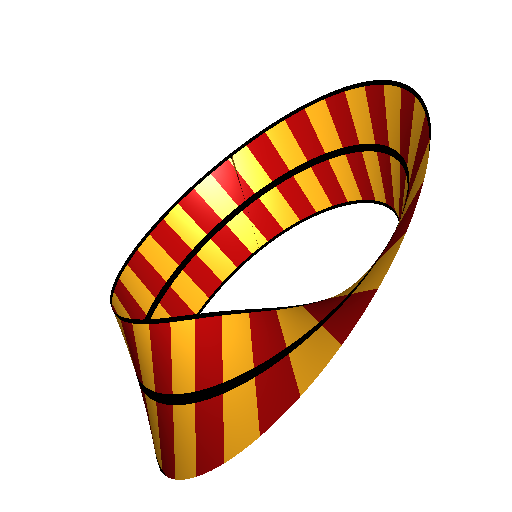}& 
		\includegraphics[width=.3\linewidth]{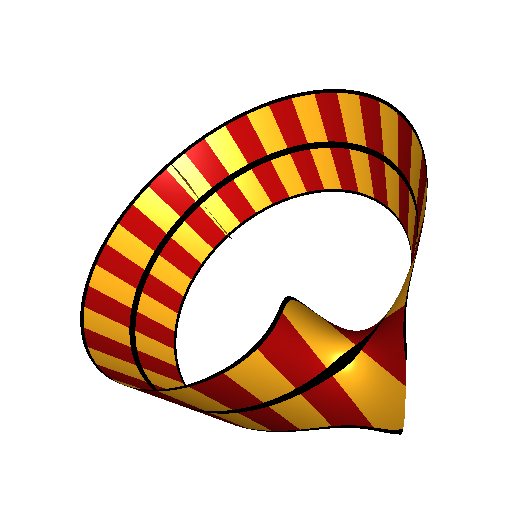}\\
		\small $B^{(0)}$ &  \small $B^{(0.4)}$  & \small $B^{(1.0)}$
	\end{tabular} 
	\\
	\begin{tabular}{ccc}
		\includegraphics[width=.3\linewidth]{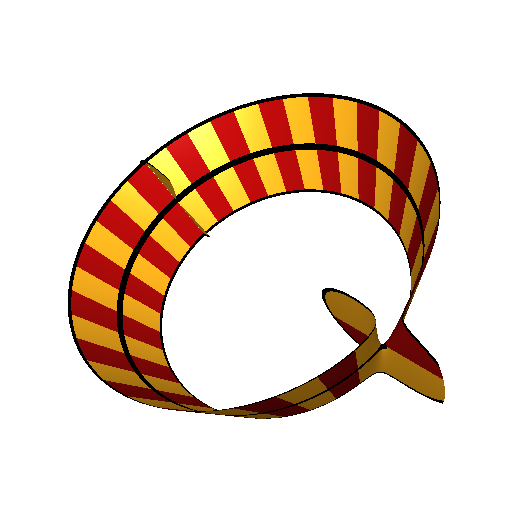} & 
		\includegraphics[width=.3\linewidth]{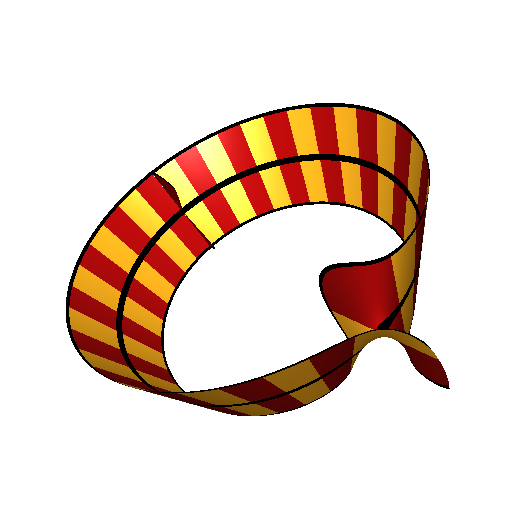} &
		\includegraphics[width=.3\linewidth]{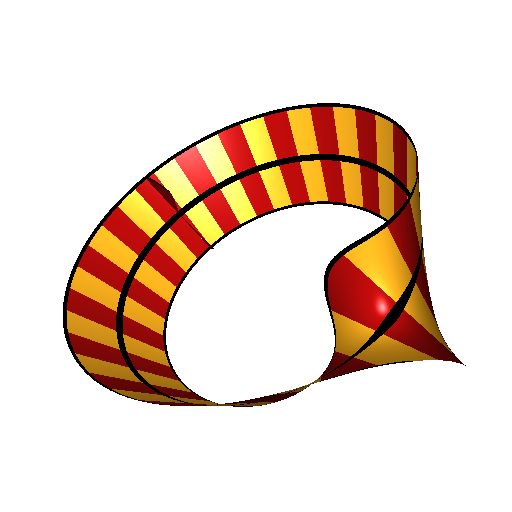}
		\\
		\small $B^{(1.15)}$ & \small $B^{(1.2)}$ & \small $B^{(1.4)}$
	\end{tabular}
	\caption{Deformation of a M\"obius Strip} 
	
\end{figure}
\end{exam}

\begin{exam} \label{Exam.4Points} {\rm As a second example, we consider a family of \textit{regular four-point blowups} 
of the real affine plane, which is indeed a modification of the example shown in Figure 9 of \cite{SS}. To this 
end, we choose $a \in [0,1]$ and consider the two pairs of polynomials $\underline{f} := (f_0,f_1)$ and 
$\underline{g} := (g_0,g_1) \in \mathbb{R}[{\bf x}, {\bf y}]^2$ given by} 
\begin{gather*}	
	f_0 = {\bf x}^2 - \frac{1}{2} {\bf y}^2 -  \frac{1}{2}, \quad f_1 = -  \frac{1}{2}{\bf x}^2+ {\bf y}^2 - 
\frac{1}{2} \mbox{ {\rm and }} \\
	g_0 = {\bf x}^2 + (a - \frac{1}{2}) {\bf y}^2-a -\frac{1}{2}, \quad g_1 = 
(a-\frac{1}{2}){\bf x}^2 + {\bf y}^2-a-\frac{1}{2}.
\end{gather*}	
{\rm Then $\det(\partial \underline{f}) = 3{\bf x}{\bf y}$ and $\det(\partial \underline{g}) = 
4(1-(a-\frac{1}{2})^2){\bf x}{\bf y}$. 
Taking ${\bf x}$-resultants, we get $\mathrm{Res}_{\bf x}(g_0,g_1) = (((a-\frac{1}{2})^2-1)(1-{\bf y}^2))^2$. 
As $(a-\frac{1}{2})^2-1 < 0$ for $a \in [0,1]$ it follows that} 
$$
Z = Z_U(\underline{f}) = Z_U(\underline{g}) = \{(1,1),(1,-1),(-1,1),(-1,-1)\}.
$$
\rm  In particular $\mathrm{det}(\partial \underline{f})(p)$ and 
$\mathrm{det}(\partial \underline{g})(p)$ are $\neq 0$ for all $p \in Z,$ so that $\underline{f}$ and 
$\underline{g}$ are regular pairs with respect to $Z$ on $U,$ with $B := \mathrm{Bl}_U(\underline{f}), 
C: = \mathrm{Bl}_U(\underline{g}) \in \mathfrak{Bl}^{\mathrm{reg}}_U(Z).$  
Moreover, 
\[
\underline{g} = \underline{f}M \mbox{ {\rm with} } 
M = 
\begin{pmatrix} 1 + \frac{2}{3}a& \frac{4}{3}a \\  
\frac{4}{3}a & 1 + \frac{2}{3}a \end{pmatrix}.
\]
{\rm so that $M \in \mathbb{R}[{\bf x},{\bf y}]^{2 \times 2}$ with $\mathrm{det}\big(M(p)\big) = 
1 + \frac{4}{3}a(1-a) > 0$ for all $p \in U.$ Setting}
$$
\widetilde{M} := \begin{pmatrix}
1 +\frac{2}{3}a{\bf t} & \frac{4}{3}a{\bf t}\\ 
\frac{4}{3}a{\bf t} & 1 +\frac{2}{3}a{\bf t} 
\end{pmatrix} 
\in \mathbb{R}[{\bf x},{\bf y}, {\bf t}]^2 \mbox{ {\rm and} } 
M^{(t)} := \begin{pmatrix}
1 +\frac{2}{3}at & \frac{4}{3}a t\\ 
\frac{4}{3}a t & 1+ \frac{2}{3} at  
\end{pmatrix} \mbox{ {\rm for all} } t \in [0,1]
$$
{\rm we get $\mathrm{det}(M^{(t)}) = (1 + \frac{2}{3}at)^2-\frac{16}{9}(at)^2 > 0$  
for all $t \in [0,1].$  Moreover, $M^{(0)} = {\bf 1}^{2 \times 2}$ and $M^{(1)} = M.$ 
So, $\big(M^{(t)}\big)_{t \in [0,1]}$ is a family which connects ${\bf 1}^{2 \times 2}$ and 
$M.$  Correspondingly $\big(\varphi^{(t)} := \varphi_{M^{(t)}}\big)_{t \in [0,1]}$ is an isotopy.  
As $\mathrm{det}(M^{(t)}) > 0$ for all $t \in [0,1]$ and 
$\mathrm{det}\big(\partial(\underline{f}M^{(t)})\big) = 
\mathrm{det}(M^{(t)})\mathrm{det}(\partial\underline{f})$ it is 
clear that
 $$
 \big(B^{(t)} = \varphi^{(t)}(B) = \mathrm{Bl}_U(\underline{f}M^{(t)})\big)_{t \in [0,1]}
 $$
is a family of regular blowups $B^{(t)} \in \mathfrak{Bl}^{\mathrm{reg}}_U(Z)$ with $B^{(0)} = B$ 
and $B^{(1)} = C.$ \\
We now choose $a = 1.$ Then looking at the conics $f_0^{(t)} = 0$ and $f_1^{(t)} = 0$ defined by the two 
polynomials} 
$$
f_0^{(t)}, f_1^{(t)} \in \mathbb{R}[{\bf x},{\bf y}] \mbox{ {\rm with} } 
\underline{f}^{(t)} := (f_0^{(t)}, f_1^{(t)}) = \underline{f}M^{(t)} \mbox{ {\rm for all} } t \in [0,1]
$$
{\rm we have the following situation: Two hyperbolas ($t=0$) are deformed to two ellipses ($t=1$) via a degeneration to a pair of lines ($t = \frac{1}{2}$).
A rough visualization of this family is shown in Figure 4.}

\begin{figure}
	\centering
	\begin{tabular}{ccc}
		\includegraphics[width=.3\linewidth]{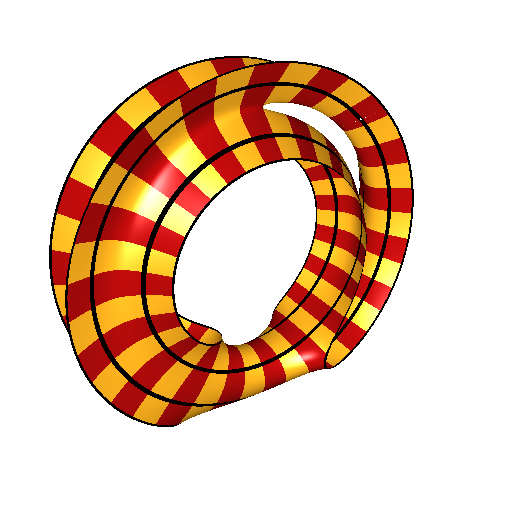}&  \includegraphics[width=.3\linewidth]{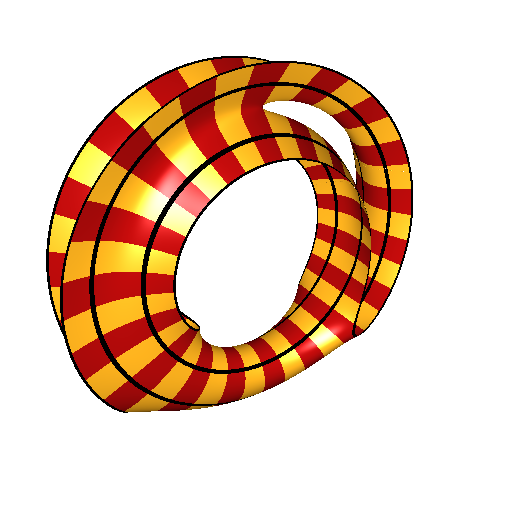}& 
		\includegraphics[width=.3\linewidth]{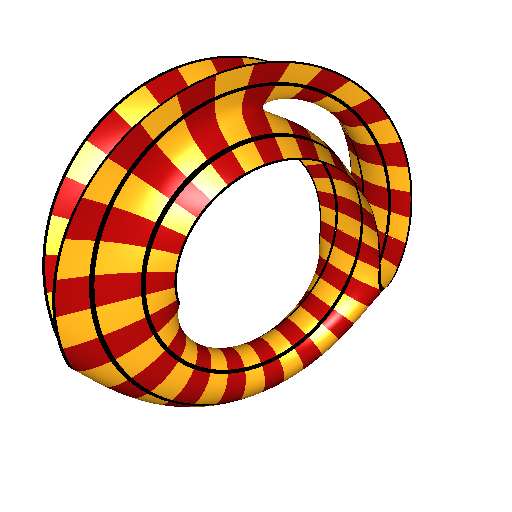}\\
		 \small $B^{(0)}$ &  \small $B^{(0.25)}$  & \small $B^{(0.5)}$
	\end{tabular} 
\\
	\begin{tabular}{cc}
		\includegraphics[width=.3\linewidth]{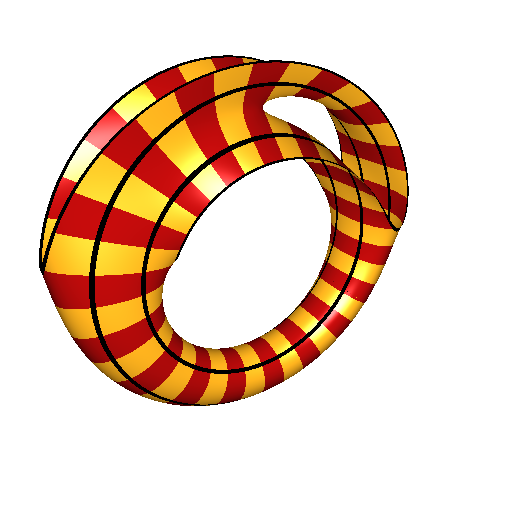} & 
		\includegraphics[width=.3\linewidth]{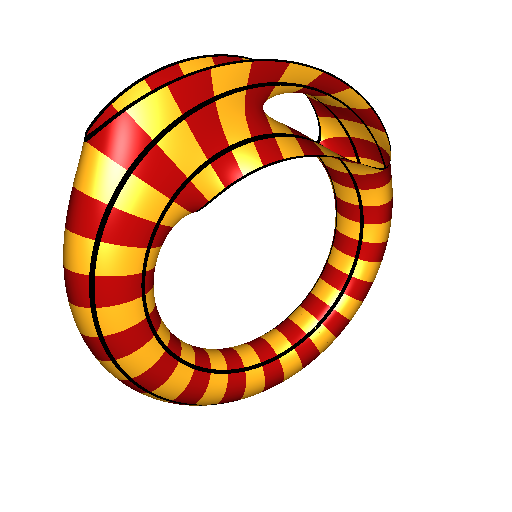}\\
		\small $B^{(0.75)}$ & \small $B^{(1)}$
	\end{tabular}
	\caption{Deformation of a regular four-point blowup} 
	
\end{figure}
\end{exam}

\begin{exam}\label{Exam.Whitney} {\rm Up to now, we have considered two families of regular blowups of the 
real affine plane. Next, we aim to consider a family of blowups, which is obtained by deforming the singular 
blowup $B := \mathrm{Bl}_U({\bf x^2},{\bf y^2}),$ whose visualization shows up as a Double Whitney Umbrella (see Figure 2(b)). We fix the matrix} 
\[
\widetilde{M} = \widetilde{M}({\bf x},{\bf y},{\bf t}) := 
\begin{pmatrix} 
	1 - {\bf t} & \frac{1}{2}{\bf t} \\ -\frac{1}{2}{\bf t} & 1 + {\bf t}
	\end{pmatrix} \in \mathbb{R}[{\bf x},{\bf y},{\bf t}]^{2 \times 2} \mbox{ {\rm with} } \mathrm{det}(\widetilde{M}) = 1 - \frac{3}{4} {\bf t}^2.
\]
{\rm For all $t \in \mathbb{R}$ we set}
\[
M^{(t)} := \widetilde{M}({\bf x},{\bf y}, t) = 
	\begin{pmatrix} 
		1 -  t & \frac{1}{2} t \\ -\frac{1}{2} t & 1 +  t
	\end{pmatrix} 
	\in \mathbb{R}[{\bf x},{\bf y}]^{2 \times 2}, \mbox{ {\rm so that} } \mathrm{det}(M^{(t)}) = 1 - \frac{3}{4} t^2.
\]
{\rm Clearly, $\mathrm{det}(M^{(t)}) > 0$ whenever $|t| < \frac{2}{3}\sqrt{3}$, so that }
$\big(\varphi_{M^{(t)}} = \varphi^{(t)}\big)_{t \in ]-\frac{2}{3}\sqrt{3},\frac{2}{3}\sqrt{3}[}$ {\rm is an isotopy of 
$U \times \mathbb{P}^1$-automorphisms.  Thus for any blowup $B = \mathrm{Bl}_U(\underline{f}) \in \mathfrak{Bl}_U(Z)$ we get a family}
$\big(B^{(t)} := \mathrm{Bl}_U(\underline{f}M^{(t)})\big)_{t \in ]-\frac{2}{3}\sqrt{3}, \frac{2}{3}\sqrt{3}[}$
 with $B^{(t)} \in \mathfrak{Bl}_U(Z) \mbox{ {\rm and} } B^{(t)} \cong B \mbox{ {\rm for all} } t \in ]-\frac{2}{3}\sqrt{3},\frac{2}{3} \sqrt{3}[.$
\begin{figure}[b]
	\centering
	\begin{tabular}{ccc}
		\includegraphics[width=.3\linewidth]{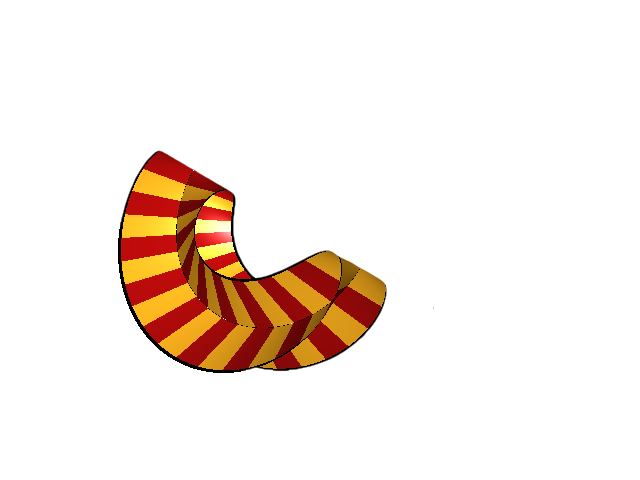}&  \includegraphics[width=.3\linewidth]{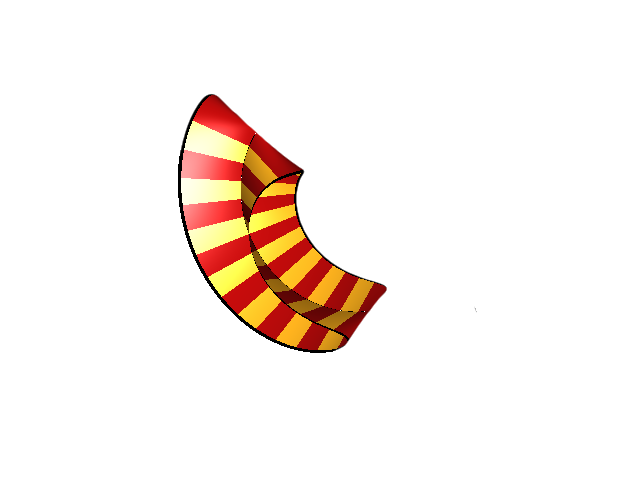}& 
		\includegraphics[width=.3\linewidth]{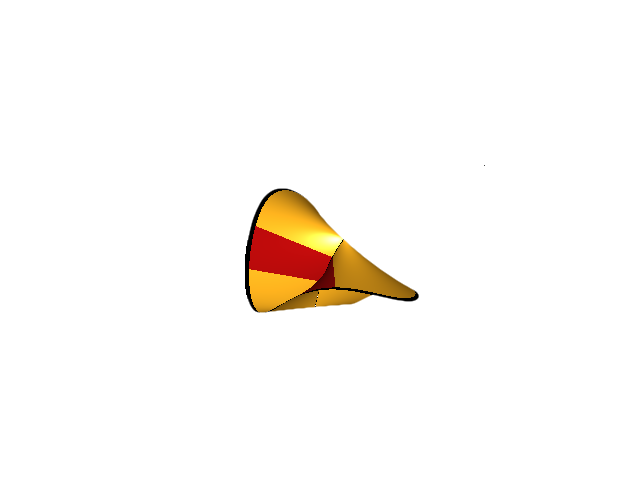}\\
		\small $B^{(0)}$ &  \small $B^{(0.5)}$  & \small $B^{(1)}$
	\end{tabular} 
	\\
	\begin{tabular}{ccc}
		\includegraphics[width=.3\linewidth]{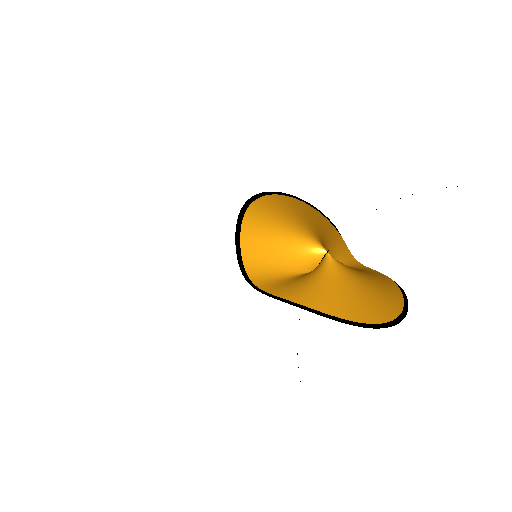} & 
		\includegraphics[width=.3\linewidth]{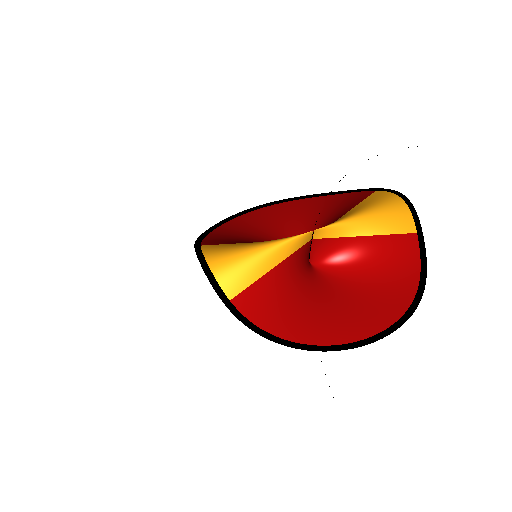} &
		\includegraphics[width=.3\linewidth]{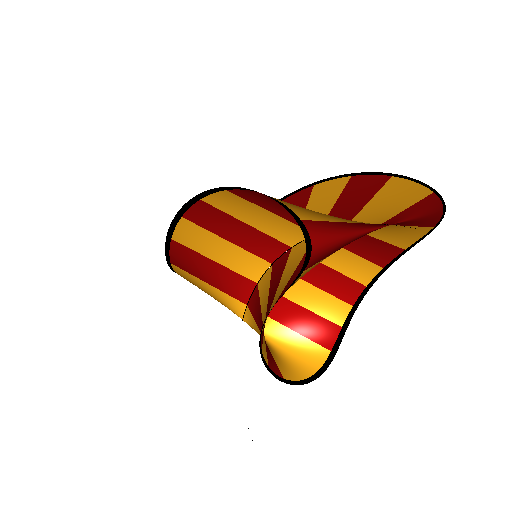}
		\\
		\small $B^{(1.1)}$ & \small $B^{(1.25)}$ & \small $B^{(4)}$
	\end{tabular}
	\caption{Deformation of a Double Whitney Umbrella}
	
\end{figure}
\end{exam}

{\rm With $f_0 = {\bf x}^2, f_1 = {\bf y}^2$ and $\underline{f}^{(t)} := \underline{f}M^{(t)}$ we then have}
$$Z := Z_U(\underline{f}^{(t)}) = \{\underline{0}\} \mbox{ {\rm for all} } t \neq \pm\frac{2}{3}\sqrt{3}.$$ 
{\rm In Figure 5, the blowups $B^{(t)}$ are visualized %$\iota\big(B^{(t)}\big) \subset \mathbb{R}^3$ 
	in $\mathbb{R}^3$ for $t = 0, \, 0.5, \, 1, \, 1.1, \, 1.25, \, 4$. Remember that $B = B^{(0)}$ is the so-called Double Whitney Umbrella.}

{\rm Note that while passing from $t = 1.1$ to $t = 1.25$ (hence by passing through the critical value 
$t = \frac{2}{3}\sqrt{3}$) the orientation of embedded blowup $B^{(t)}$ swaps. 
Observe also, that the fiber 
$\pi^{-1}_{U, \underline{f}^{(t)}}(\underline{0}) = \{\underline{0}\} \times \mathbb{P}^1$ 
of $B^{(t)}$ over $\underline{0}$ is visualized by the same circle for all 
$t \neq \pm\frac{2}{3}\sqrt{3}$ and that the corresponding 
set of limit points $\mathcal{L}_{\underline{0}}(B^{(t)})$ is visualized by an arc on this circle, whose length depends on $t.$ 
Near to the degeneration value $t = \frac{2}{3} \sqrt{3}$ we enlarged the scale of our visualization in order 
to improve the picture of the details. Therefore the coloring appears larger for the last three values 
of $t$.}

\section{Structure and Classification of Regular Embedded Blowups}

 \noindent{\rm  \bf{Equality of Embedded Blowups.}} {\rm  Let all 
 	notations be as in the Introduction. Our first aim is to make clear, when the embedded blowups of $U$ with respect 
 	to two pairs of polynomials are equal}}.

\begin{prop} \label{Proposition Equal}  Let $Z, W \subset U$ be two finite sets and let 
$\underline{f} = (f_0,f_1), \underline{g} = (g_0,g_1) \in \mathbb{R}[{\bf x},{\bf y}]^2$ such that 
$Z_{\overline{U}}(\underline{f}) = Z$ and $Z_{\overline{U}}(\underline{g}) = W.$
\begin{itemize}
	\item[\rm{(a)}] Then $\mathrm{Bl}_U(\underline{f}) = \mathrm{Bl}_U(\underline{g})$ if and only 
	if $f_0g_1 = f_1g_0.$
	\item[\rm{(b)}] If $f_0$ and $f_1$ have no common divisor, the pair $\underline{f}$ is uniquely 
	determined by $\mathrm{Bl}_U(\underline{f})$ up to multiplication with a non-zero constant.  
\end{itemize}             
\end{prop} 

\begin{proof} (a):  As $U \setminus (Z \cup W) \neq \emptyset,$ we have $\underline{f}, \underline{g} \neq (0,0).$ 
Assume first, that $\mathrm{Bl}_U(\underline{f}) = \mathrm{Bl}_U(\underline{g}).$ Then clearly 
$\mathrm{Bl}_U(\underline{f}) \setminus \big((Z \cup W) \times \mathbb{P}^1\big) = \mathrm{Bl}_U(\underline{g}) 
\setminus((Z \cup W) \times \mathbb{P}^1\big).$ As $\mathrm{E}_{U}(\underline{f}),  \mathrm{E}_{U}(\underline{g}) 
\subseteq (Z \cup W) \times \mathbb{P}^1$ (see (1.4)(c)), $\mathrm{Bl}_U(\underline{f}) = 
\mathrm{Bl}^\circ_U(\underline{f}) \, \dot{\cup} \, \mathrm{E}_{U}(\underline{f})$ and 
$\mathrm{Bl}_U(\underline{g}) = \mathrm{Bl}^\circ_U(\underline{g})\, \dot{\cup} \, \mathrm{E}_{U}(\underline{g}),$ 
it follows that $\mathrm{Bl}^\circ_U(\underline{f}) \setminus 
\big((Z \cup W) \times \mathbb{P}^1\big) = \mathrm{Bl}^\circ_U(\underline{g})\setminus \big((Z \cup W) 
\times \mathbb{P}^1\big).$ But according to the definition of the operation $\mathrm{Bl}^{\circ}_U(\bullet)$ of 
taking open kernels (see (1.4)(b)), this means that the graphs of the two restricted maps 
$\varepsilon_{U,\underline{f}}\upharpoonright, \varepsilon_{U,\underline{g}}\upharpoonright: 
U \setminus (Z \cup W) \longrightarrow \mathbb{P}^1$ (and thus these restricted maps themselves) coincide. So, 
for all $p \in U \setminus (Z \cup W)$ it holds $(f_0(p) : f_1(p)) = (g_0(p) : g_1(p)).$ Assume now, that 
$f_0 \neq 0.$ Then, there is a dense open subset $V \subseteq U \setminus (Z \cup W)$ such that $f_0(p) \neq 0$ 
and $(f_0(p) : f_1(p)) = (g_0(p):g_1(p))$ for all $p \in V.$ As $V \neq \emptyset$ is open in $\mathbb{R}^2,$ it 
follows that the two rational functions $\frac{f_1}{f_0}, \frac{g_1}{g_0} \in \mathbb{R}({\bf x},{\bf y})$ are 
defined and coincide and hence that $f_0g_1 = f_1g_0.$ If $f_0 = 0$ we have $f_1 \neq 0$ and hence may conclude 
similarly.\\
Assume now that $f_0g_1 = f_1g_0.$ Suppose first, that $f_0 \neq 0.$ Then $g_0 =0$ would imply $g_1 =0$ 
and hence the contradiction that $\underline{g} = (0,0).$ So, we have $g_0 \neq 0.$ Therefore we find a 
dense open subset $V \subseteq U \setminus (Z \cup W)$ such that for all $p \in V$ we have 
$f_0(p), g_0(p) \neq 0$ and $(f_0(p):f_1(p)) = (g_0(p) : g_1(p)).$ This means, that the two restricted 
maps $\varepsilon_{U,\underline{f}}\upharpoonright, \varepsilon_{U,\underline{g}}\upharpoonright: 
V \longrightarrow \mathbb{P}^1$ coincide and hence have the same graph 
$$
S := \{\big(p,(f_0(p):f_1(p)\big) = \big(p,(g_0(p):g_1(p))\big) \mid p \in V\} \subseteq \mathrm{Bl}^\circ_U(\underline{f}) \cap \mathrm{Bl}^\circ_U(\underline{g}).
$$
As $V$ is open and dense in $U \setminus Z,$ the isomorphism $\pi_{U,\underline{f}}\upharpoonright: \mathrm{Bl}^\circ_U(\underline{f}) \stackrel{\cong}{\longrightarrow} U \setminus Z$ yields that $S :=  \big(\pi_{U,\underline{f}}\upharpoonright\big)^{-1}(V)$ is open and dense in $\mathrm{Bl}^\circ_U(\underline{f})$ with respect to the strict (e.g. metric) topology of $U \times \mathbb{P}^1,$ so that $S$ and $\mathrm{Bl}^\circ_U(\underline{f})$ have the same strict closure. The same applies for $S$ and $\mathrm{Bl}^\circ_U(\underline{g}).$ Therefore, $\mathrm{Bl}^\circ_U(\underline{f})$ and $\mathrm{Bl}^\circ_U(\underline{g})$ have the same strict closure. As the Zariski topology is coarser than the strict topology, it follows, that these two sets also have the same Zariski closure, and hence that $\mathrm{Bl}_U(\underline{f}) = \mathrm{Bl}_U(\underline{g})$ (see (1.3)(a)).\\
(b): Assume neither $f_0$ and $f_1$ nor $g_0$ and $g_1$ have a common divisor and that $\mathrm{Bl}_U(\underline{f}) = \mathrm{Bl}_U(\underline{g}).$ By statement (a) we get $f_0g_1 = f_1g_0.$ As $\mathbb{R}[{\bf x},{\bf y}]$ is factorial we find some $c \in \mathbb{R} \setminus \{0\}$ such that $\underline{g} = c\underline{f}.$  
\end{proof}

\noindent{\rm \bf{Structure of Regular Embedded Blowups.}} {\rm We next prove a structure result for regular blowups.}

\begin{prop}\label{Proposition Structure} Let $B \in \mathfrak{Bl}_U^{\mathrm{reg}}(Z).$ Then $B$ is a smooth real algebraic hyper-surface in $U \times \mathbb{P}^1.$
\end{prop}

\begin{proof} {\rm Let $\underline{f} = (f_0,f_1) \in \mathbb{R}[{\bf x},{\bf y}]^2$ be a regular pair 
on $U$ with respect to $Z,$ such that $B = \mathrm{Bl}_U(\underline{f}).$ Let 
$h := {\bf z}_0f_1({\bf x},{\bf y}) - {\bf z}_1f_0({\bf x},{\bf y}) \in \mathbb{R}[{\bf x},{\bf y},{\bf z}_0,{\bf z}_1].$}
{\rm If $(x,y) \in U \setminus Z$ and $(u:v) \in \mathbb{P}^1$ we have $h(x,y,u,v) = 0$ if and only if 
$\big((x,y),(u:v)\big) \in B^\circ,$  so that $B^\circ = \{\big((x,y), (u:v)\big) \in (U \setminus Z) \times \mathbb{P}^1 \mid h(x,y,u,v) = 0\}.$ 
Passing to Zariski closures we get (see (1.3)(a)) $B = \{((x,y),(u:v)) \in U \times \mathbb{P}^1 \mid h(x,y,u,v) = 0\}.$ It remains to show, that 
$$
\big(\frac{\partial h}{\partial {\bf x}}(x,y,u,v), \frac{\partial h}{\partial {\bf y}}(x,y,u,v), \frac{\partial h}{\partial {\bf z}_0}(x,y,u,v), 
\frac{\partial h}{\partial {\bf z}_1}(x,y,u,v)\big) \neq \underline{0},
$$
whenever $\big((x,y),(u:v)\big) \in B.$ As $\frac{\partial h}{\partial {\bf z}_0} = f_1$ and $\frac{\partial h}{\partial {\bf z}_1} = - f_0,$ this is clear 
if $p := (x,y) \notin Z.$ If $p = (x,y) \in Z,$ we have $\mathrm{rank}\big((\partial \underline{f})(p)\big) = 2,$ and $(u,v) \neq (0,0)$ shows that
$$
\big(\frac{\partial h}{\partial {\bf x}}(x,y,u,v), \frac{\partial h}{\partial {\bf y}}(x,y,u,v)\big) = 
\big(u\frac{\partial f_1}{\partial {\bf x}}(p) - v\frac{\partial f_0}{\partial {\bf x}}(p), 
u\frac{\partial f_1}{\partial {\bf y}}(p) - v\frac{\partial f_0}{\partial {\bf y}}(p)\big) \neq \underline{0}.
$$}
\end{proof}

\noindent{\rm \bf{Reduced and Strongly Regular Pairs and Application to Sign Distributions.}} {\rm The remaining part of this section is devoted to the Classification Theorem mentioned in (1.11) and hence to the solution of the Classification Problem (1.0)(b) for regular embedded blowups. We first will introduce two special types of regular pairs of polynomials.} %and relate these to the sign distribution map which was mentioned already in (1.10).}

\begin{prop defn}\label{Lemma and Definition BAP2} Let $B \in \mathfrak{Bl}_U^{\mathrm{reg}}(Z).$ Then, there 
is a regular pair $\underline{f} = (f_0, f_1) \in \mathbb{R}[{\bf x},{\bf y}]^2,$ with respect to $Z$ on $U,$  
 unique up to multiplication with a non-zero constant -- and called a \textit{reduced regular pair for $B$} -- 
such that
\begin{itemize}
\item[\rm{(a)}] $f_0$ and $f_1$ have no common divisor.
\item[\rm{(b)}] $\mathrm{Bl}_U(\underline{f}) = B.$
\item[\rm{(c)}] If $\underline{g} = (g_0,g_1) \in \mathbb{R}[{\bf x},{\bf y}]^2$ is a regular pair with respect to $Z$ on $U$ with $B = \mathrm{Bl}_U(\underline{g}),$ then there is a unique polynomial $h \in \mathbb{R}[{\bf x},{\bf y}]$ such that $\underline{g} = h\underline{f}.$ Moreover, in this situation
\begin{itemize}
\item[\rm{(1)}] $h(p) \neq 0$ for all $p \in U.$
\item[\rm{(2)}] $\mathrm{sgn}\big(\mathrm{det}(\partial \underline{g}(p))\big) = \mathrm{sgn}\big(\mathrm{det}(\partial \underline{f}(p))\big)$ for all $p \in Z.$
\end{itemize}
\end{itemize}
\end{prop defn}

\begin{proof} By our definition (1.9) of 
$\mathfrak{Bl}_U^{\mathrm{reg}}(Z)$ 
we may write $B = \mathrm{Bl}_U(\underline{g})$ where 
$(g_0,g_1) = \underline{g} \in \mathbb{R}[{\bf x},{\bf y}]^2$ 
is a regular pair with respect to $Z$ on $U$. Now, choose any such  pair $\underline{g}.$ Let 
$h \in \mathbb{R}[{\bf x},{\bf y}]$ be a greatest common divisor of $g_0$ and $g_1$ and let 
$\underline{f} = (f_0, f_1) \in \mathbb{R}[{\bf x},{\bf y}]^2$ be such that $\underline{g} = h\underline{f}.$ 
By Proposition~\ref{Proposition Equal} (a) we have $\mathrm{Bl}_U(\underline{f}) = 
\mathrm{Bl}_U(\underline{g}) = B.$ The Leibniz product rule for derivatives gives
\begin{itemize}
\item[\rm{(@)}] $\partial \underline{g} = \partial (h\underline{f}) = h \partial \underline{f} + \begin{pmatrix} f_0\frac{\partial h}{\partial {\bf x}} & f_1\frac{\partial h}{\partial {\bf x}} \\ {} & {} \\ f_0\frac{\partial h}{\partial {\bf y}} & f_1\frac{\partial h}{\partial {\bf y}} \end{pmatrix}.$
\end{itemize}
 
We claim that $h(p) \neq 0$ for all $p \in \overline{U}.$ If we assume to the contrary that $h(p) = 0$ 
for some $p \in \overline{U},$ it would follow -- by $\underline{g} = h \underline{f}$ -- that 
$p \in Z_{\overline{U}}(\underline{g}) = Z.$ But then by (@) the matrix $\partial \underline{g}(p)$ would be of rank at most 
$1,$ which contradicts the fact that $\underline{g}$ is regular with respect to $Z$ on $U.$ In 
particular we now get that $Z_{\overline{U}}(\underline{f}) = Z.$ 
Now, another use of (@) gives that for all $p \in Z$ we have $h(p)\big(\partial \underline{f}\big)(p) = 
\partial \underline{g}(p)$ and hence $h(p)^2\mathrm{det}\big(\partial \underline{f}\big) (p) = 
\mathrm{det}\big(\partial \underline{g}\big)(p) \neq 0,$ thus $\big(\partial \underline{f}\big)(p) \neq 0.$ 
This shows that $\underline{f}$ is a regular pair with respect to $Z$ on $U$ by definition.  
Finally, a further use of (@) shows that $\mathrm{sgn}\big(\mathrm{det}(\partial \underline{g}(p))\big) = 
\mathrm{sgn}\big(\mathrm{det}(\partial \underline{f}(p))\big)$ for all $p \in Z.$ As $f_0$ and 
$f_1$ have no common divisor, the stated uniqueness of the pair $\underline{f}$ follows by 
Proposition~\ref{Proposition Equal} (b). 
\end{proof}

\begin{defn rem}\label{Definition sgn} {\rm Let $B \in \mathfrak{Bl}_U^{\mathrm{reg}}(Z)$ and let $p \in Z.$ 
We write $B = \mathrm{Bl}_U(\underline{g}),$ where $\underline{g} \in \mathbb{R}[{\bf x},{\bf y}]^2$ is a regular 
pair with respect to $Z$ on $U.$ Then, by Lemma and Definition~\ref{Lemma and Definition BAP2} (c) (2) it is immediate, that 
$\mathrm{sgn}\big(\mathrm{det}(\partial \underline{g}(p))\big)$ depends only on the blowup $B$ and not on the 
chosen defining pair $\underline{g}.$  
This allows to define a map (see (1.10))}
 $$\mathrm{sgn}_B = \mathrm{sgn}_{\underline{g}}: Z \longrightarrow \{\pm 1 \} \mbox{ {\rm given by }} p \mapsto \mathrm{sgn}_B(p) :=\mathrm{sgn}\big(\mathrm{det}(\partial \underline{g}(p))\big) \mbox{ {\rm for all }} p \in Z.$$
{\rm We call this map the \textit{sign distribution} of $B.$}
\end{defn rem}
 
We now define a notion related to the complex affine plane. We like to do this, as in this way we get 
a stronger result (see Lemma~\ref{Lemma EI1}). 
 
\begin{defn rem}\label{Definition EI2} {\rm (A) Let $Z = \{p_i = (x_i,y_i) \mid i=1,2,\ldots,n\} \subset U, 
(p_i \neq p_j$ for all $i \neq j).$ A pair $\underline{f} = (f_0,f_1) \in \mathbb{R}[{\bf x},{\bf y}]^2$ 
is called \textit{strongly regular with respect to $Z$ (on $U$)}, if it satisfies the following equivalent 
requirements:}
 
\begin{itemize}
\item[\rm{(i)}] $\mathbb{C}[{\bf x},{\bf y}]f_0 + \mathbb{C}[{\bf x},{\bf y}]f_1 = 
\bigcap_{i=1}^n\big(\mathbb{C}[{\bf x},{\bf y}]({\bf x}-x_i) + \mathbb{C}[{\bf x},{\bf y}]({\bf y}-y_i)\big).$
\item[\rm{(ii)}] $\mathbb{C}[{\bf x},{\bf y}]f_0 + \mathbb{C}[{\bf x},{\bf y}]f_1 = 
I_{\mathbb{A}^2_{\mathbb{C}}}(Z) := \{f \in \mathbb{C}[{\bf x},{\bf y}] \mid f(p) = 0, \forall p \in Z\}.$
\end{itemize}
 
{\rm (B) Assume that $\underline{f}\in \mathbb{R}[{\bf x},{\bf y}]^2$ is a strongly regular pair with respect to 
$Z.$ Then, it is easy to see:}
 
\begin{itemize}
\item[\rm{(a)}] $\underline{f}$ {\rm is a regular pair with respect to $Z$ on $U$ in the sense of (1.14).}
\item[\rm{(b)}] $\underline{f}$ {\rm is a reduced regular pair for $B := \mathrm{Bl}_U(\underline{f})$ in the 
sense of Lemma and Definition~\ref{Lemma and Definition BAP2}.} 
\item[\rm{(c)}] $\mathbb{R}[{\bf x},{\bf y}]f_0 + \mathbb{R}[{\bf x},{\bf y}]f_1 = I_{\mathbb{A}^2_{\mathbb{R}}}(Z) := 
\{g \in \mathbb{R}[{\bf x},{\bf y}] \mid g(p) = 0, \forall p \in Z\}$.
\end{itemize}
\end{defn rem}

\begin{lem}\label{Lemma EI1} Let $n > 0,$ let $Z := \{p_1,p_2,\ldots,p_n\}$ a set of pairwise different points 
with $p_i := (x_i,y_i) \in U$ for $i=1,2,\ldots,n$. Let $\chi:Z \longrightarrow \mathbb{R}\setminus \{0\}$ be a 
map. Then, there is a strongly regular pair $ \underline{f} =  (f_0, f_1) \in \mathbb{R}[{\bf x},{\bf y}]^2$ 
with respect to $Z$ such that $\mathrm{det}\big(\partial \underline{f}(p)\big) = \chi(p)$ for all $p \in Z.$
\end{lem}

\begin{proof} We shall give the proof under the assumption that 
$x_i \neq x_j$ for all $i,j \in \{1,2,\ldots,n\}$ with $i \neq j.$ Should this requirement not be satisfied, 
we first subject $\mathbb{R}^2$ to a general transformation $T \in \mathrm{SL}_2(\mathbb{R}),$ such that our 
requirement is satisfied -- and keep in mind that this does not affect our Jacobian determinants. Then, we perform our 
proof as below and finally apply the transformation $T^{-1}.$        
\noindent{\rm We set} 
$$f_0 := \prod_{i=1}^n ({\bf x} - x_i) \in \mathbb{R}[{\bf x}] \mbox{ {\rm and} } f_1 := 
h({\bf x})\big({\bf y} - g({\bf x})\big),$$ 
{\rm where $g({\bf x}), h({\bf x}) \in \mathbb{R}[{\bf x}]$ are the uniquely determined polynomials of degree 
$\leq  n-1$ which respectively satisfy}
$$g(x_i) = y_i \mbox{ {\rm and} } h(x_i) = \frac{\chi(p_i)}
%{\prod_{j\neq i}(x_i-x_j)} 
{\prod_{j\in\{1, \ldots,n\}\setminus\{i\}}(x_i-x_j)}
\mbox{ {\rm for all} } i=1,2,\ldots,n.$$
Observe also, that 
$$\frac{\partial f_0}{\partial{\bf x}}(p_i) = \prod _{j\in\{1, \ldots,n\}\setminus\{i\}}(x_i-x_j)\;  \mbox{ {\rm and} } \frac{\partial f_1}{\partial {\bf y}}(p_i) = h(x_i) \mbox{ {\rm for all} } i = 1,2,\ldots,n.$$ 
{\rm Now, for all $i=1,2,\ldots,n$ we obtain:}
\begin{align*}
 \mathrm{det}\big(\partial \underline{f}(p_i)\big) &= \mathrm{det}\begin{pmatrix}\frac{\partial f_0}{\partial{\bf x}}(p_i) & \frac{\partial f_1}{\partial{\bf x}}(p_i)\\ {} & {} \\                                      
                               \frac{\partial f_0}{\partial{\bf y}}(p_i) & \frac{\partial f_1}{\partial{\bf y}}(p_i)\end{pmatrix} = 
                               \\&{}\\
                           &= \mathrm{det}\begin{pmatrix}\prod _{j\neq i}(x_i-x_j) & \frac{\partial \big(h({\bf x})({\bf y} - g({\bf x}))\big)}{\partial{\bf x}}(p_i) \\ {} & {} \\                               
                               0  & h(x_i)\end{pmatrix}=\chi(p_i). 
\end{align*}
{\rm  Therefore $\mathrm{det}\big(\partial\underline{f}(p_i)\big) = \chi(p_i)$ for all 
$i=1,2,\ldots,n.$ \\
It is immediate to see, that $Z =\{p_1,p_2,\ldots,p_n\}$ is precisely the set
$Z_{\mathbb{C}^2}(\underline{f})$ of common zeros of the two polynomials 
$f_0, f_1 \in \mathbb{C}[{\bf x},{\bf y}]$ in $\mathbb{C}^2.$ 
As $\mathrm{det}\big(\partial \underline{f}(p_i)\big) = \chi(p_i) \neq 0$ for all $i \in \{1,2\ldots,n\}$ it follows
by the Jacobian Criterion, that $\mathbb{C}[{\bf x},{\bf y}]f_0 + \mathbb{C}[{\bf x},{\bf y}]f_1$ is 
reduced and hence is the vanishing ideal $I_{\mathbb{A}^2_{\mathbb{C}}}(Z)$ of $Z$ in $\mathbb{C}[{\bf x},{\bf y}].$
So $\underline{f}$ is strongly regular with respect to $Z$ on $U.$}
\end{proof}

\noindent{\rm \bf{The Classification Result.}} Now we will establish the Isomorphy Criterion we are heading 
for in this section, and hence solve the Classification Problem mentioned under (1.0) (b). We first shall prove 
two auxiliary results  whose proofs are straight forward. As both of them are crucial for the proof 
of our Classification Theorem, we include their proofs for the reader's convenience.

\begin{lemma}\label{Lemma EI2}  Let $\underline{f} = (f_0,f_1), 
\underline{g} = (g_0,g_1) \in \mathbb{R}[{\bf x},{\bf y}]^2$ be two pairs such that 
$\mathrm{Z}_U(\underline{f}) = \mathrm{Z}_U(\underline{g}) =Z. $  Assume that there 
exists a matrix $N \in \mathbb{R}[{\bf x},{\bf y}]^{2\times 2}$ such that $\underline{g} = \underline{f}N.$ 
Moreover, for each $\gamma \in \mathbb{R}[{\bf x},{\bf y}]$ we set 
$$N_\gamma := N + \gamma\begin{pmatrix}g_1 f_1 & -g_0 f_1\\ -g_1 f_0 & g_0 f_0\end{pmatrix}.$$
Then, it holds
\begin{itemize}
\item[\rm{(a)}] $N_\gamma (p) = N(p)$ for all $p \in Z.$
\item[\rm{(b)}] $\underline{g} = \underline{f}N_\gamma.$
\item[\rm{(c)}] $\mathrm{det}(N_\gamma) = \mathrm{det}(N) + \gamma\big(g_0^2 + g_1^2\big).$
\item[\rm{(d)}] If $\mathrm{det}\big(N(p)\big) > 0$ for all $p \in Z,$ then, there is some 
$b \in \mathbb{R}_{>0}$ such that  $\mathrm{det}\big(N_\gamma(p)\big) > 0$ for all $p \in U$ and all 
$\gamma \in \mathbb{R}[{\bf x},{\bf y}]$ with 
$\mathrm{inf}\{\gamma(p)  \mid p \in U\} > b.$ 
\end{itemize} 
\end{lemma}

\begin{proof} {\rm Statements (a) and (b) are immediate. To prove statement (c) we write} 
$$N = \begin{pmatrix} N_{11} & N_{12}\\ N_{21} & N_{22}\end{pmatrix}$$
{\rm On use of the column bi-linearity of the determinant and as} 
$$\mathrm{det}\begin{pmatrix} f_1 & N_{12}\\ -f_0 & N_{22}\end{pmatrix} = g_1 \mbox{ {\rm and} }  \mathrm{det}\begin{pmatrix}N_{11} & -f_1\\ N_{21} & f_0 \end{pmatrix} = g_0,$$
{\rm we get indeed}
\begin{align*} 
\mathrm{det}(N_\gamma) &= \mathrm{det}\big(N+ \gamma \begin{pmatrix}g_1 f_1 &-g_0 f_1\\ -g_1 f_0 &  g_0 f_0\end{pmatrix}\big) 
=\mathrm{det}\begin{pmatrix} N_{11} + \gamma g_1 f_1 & N_{12}-\gamma g_0 f_1 \\ N_{21}-\gamma g_1 f_0  & N_{22}+\gamma g_0 f_0\end{pmatrix} = \\
&= \mathrm{det}\begin{pmatrix}N_{11}&N_{12}\\N_{21}&N_{22}\end{pmatrix}+\mathrm{det}\begin{pmatrix}\gamma g_1 f_1 & N_{12}\\-\gamma g_1 f_0 &N_{22}\end{pmatrix} +\\ 
&+ \mathrm{det}\begin{pmatrix}N_{11}& -\gamma g_0 f_1\\ N_{21} & \gamma g_0 f_0\end{pmatrix} + \mathrm{det}\begin{pmatrix}\gamma g_1 f_1 & -\gamma g_0 f_1\\ - \gamma g_1 f_0 & \gamma g_0 f_0\end{pmatrix} = \\
&= \mathrm{det}(N) + \gamma g_1 \mathrm{det}\begin{pmatrix} f_1 & N_{12}\\ -f_0 & N_{22}\end{pmatrix} + \gamma g_0 \mathrm{det}\begin{pmatrix}N_{11} & -f_1\\ N_{21} & f_0 \end{pmatrix} + 0 = \\
&= \mathrm{det}(N) + \gamma g_1^2+\gamma g_0^2 = \mathrm{det}(N) + \gamma\big(g_0^2 + g_1^2\big). 
\end{align*} 
{\rm It remains to show statement (d). So, assume that $\mathrm{det}\big(N(p)\big) > 0$ for all $p \in Z.$ We have to show that there is some constant $b \in \mathbb{R}_{>0}$ such that $\mathrm{det}\big(N_\gamma(p)\big) > 0$ for all $p \in U$ and all constants $\gamma > b.$ As $\mathrm{det}\big(N(p)) > 0$ for all $p \in Z,$ there is some open set $W \subset U$ such that $Z \subset W$ and $\mathrm{det}\big(N(p)) > 0$ for all $p \in W.$ It follows by statement (a) and (c) that} 
$$\mathrm{det}\big(N_\gamma(p)) > 0 \mbox{ {\rm for all} } p \in W \mbox{ {\rm and all} } \gamma > 0. $$
{\rm As $U$ is bounded and $\mathrm{Z}_{\mathbb{R}^2}(\underline{g})$ does not contain any points of the boundary of $U$ it follows that there is some $c > 0$ such that $g_0(p)^2 + g_1(p)^2 > c$ for all $p \in U \setminus W.$ As $U$ is bounded, there is some $C > 0$ such that $\mathrm{det}\big(N(p)\big) \geq -C$ for all $p\in U$. If $\gamma > b:= \frac{C}{c}$ it follows that} 
$$\mathrm{det}\big(N_{\gamma}(p)\big) \geq \mathrm{det}\big(N(p)\big) + b\big(g_0(p)^2 + g_1(p)^2\big) > 0 \mbox{ {\rm for all} } p \in U \setminus W,$$
{\rm and hence $\mathrm{det}\big(N_{\gamma}(p)\big) > 0$ for all $p \in U.$ } 
\end{proof}

\begin{lemma}\label{Proposition EI1} Let $\underline{f} =(f_0,f_1), \underline{g} = (g_0,g_1)\in \mathbb{R}[{\bf x},{\bf y}]^2$ be two pairs of polynomials such that $\underline{f}$ is strongly regular with respect $Z$ and $\underline{g}$ is regular with respect to $Z$ on $U$ and consider the two blowups $B := \mathrm{Bl}_U(\underline{f}), C:= \mathrm{Bl}_U(\underline{g}) \in \mathfrak{Bl}_U^{\mathrm{reg}}(Z).$ Then, the following statements are equivalent:
\begin{itemize}
\item[\rm{(i)}] $\mathrm{sgn}_C = \mathrm{sgn}_{B}.$ 
\item[\rm{(ii)}] There is a matrix  $M \in \mathbb{R}[{\bf x},{\bf y}]^{2 \times 2}$ such that $\mathrm{det}\big(M(p)\big) > 0$ for all $p \in U$ and $\underline{g} = \underline{f}M.$ 
\end{itemize}
\end{lemma}

\begin{proof} {\rm (i) $\Rightarrow$ (ii): Assume that statement (i) holds. As $g_0, g_1 \in I_{\mathbb{A}^2_{\mathbb{R}}}(Z),$ it follows by Definition and Remark~\ref{Definition EI2}(B)(c), that there is a matrix}
$$ N = \begin{pmatrix} N_{11} & N_{12} \\ N_{21} & N_{22} \end{pmatrix} = \begin{pmatrix} N_{1 \bullet}\\N_{2 \bullet}\end{pmatrix} \in \mathbb{R}[{\bf x},{\bf y}]^{2\times 2} \mbox{ {\rm with} } \underline{g} = \underline{f}N.$$
{\rm By our assumption we have $\mathrm{sgn}\big(\mathrm{det}(\partial \underline{g}(p))\big) = \mathrm{sgn}_C(p) = \mathrm{sgn}_B(p) = \mathrm{sgn}\big(\mathrm{det}(\partial \underline{f}(p))\big)$ for all $p \in Z.$ Moreover, by the Leibniz product rule for derivatives we have}
\begin{itemize}
\item[\rm{(@@)}] $\quad \quad \quad \quad  \partial \underline{g} = \partial(\underline{f}N) = \partial\underline{f} \cdot N + f_0\cdot \partial N_{1 \bullet} + f_1 \cdot \partial N_{2 \bullet}.$
\end{itemize}
{\rm As $\underline{f}(Z) = 0$ it follows that $\mathrm{det}(\partial \underline{g}(p)) = \mathrm{det}(\partial \underline{f}(p)) \cdot \mathrm{det}\big(N(p)\big)$ and hence $\mathrm{det}\big(N(p)\big) > 0$ for all $p \in Z.$
Now, by Lemma~\ref{Lemma EI2} (c), there is some $\gamma \in \mathbb{R}_{>0}$ such that the matrix $M : = N_\gamma \in \mathbb{R}[{\bf x},{\bf y}]^{2 \times 2}$ satisfies $\mathrm{det}\big(M(p)\big) > 0$ for all $p \in U.$ Moreover, Lemma~\ref{Lemma EI2} (b) yields that $\underline{g} = \underline{f}M.$}\\
{\rm (ii) $\Rightarrow$ (i):  Assume that statement (ii) holds. The 
Leibniz product rule for derivatives (see formula (@@) above) and the fact that $f_0(p) = f_1(p) = 0$ for all $p \in Z$ give} 
$$
\partial \underline{g}(p) = \partial(\underline{f}M)(p) = \partial\underline{f} (p)\cdot M(p) 
\mbox{ for all } p \in Z.
$$
{\rm Taking determinants and observing that $\mathrm{det}\big(M(p)\big) > 0$ for all $p 
\in Z$ we get statement (i). }
\end{proof}

{\rm Now, we are ready to prove the main result of this section (cf. (1.11)).}

\begin{thm}\label{Theorem EI1} (Classification of Regular Embedded Blowups)
\begin{itemize}
\item[\rm{(a)}]  For each function $\sigma:Z \longrightarrow \{+1,-1\}$ there is a regular embedded blowup 
$B \in \mathfrak{Bl}_U^{\mathrm{reg}}(Z)$ such that $\mathrm{sgn}_B = \sigma.$
\item[\rm{(b)}] Let $B, C \in \mathfrak{Bl}_U^{\mathrm{reg}}(Z).$ Then $B$ and $C$ are relatively oriented embedded
isomorphic if and only if they have the same sign distribution. Hence, for short:  $B \cong C $ if and only if 
$\mathrm{sgn}_B = \mathrm{sgn}_C.$ 
\item[\rm{(c)}] There are precisely $2^{\#Z}$ isomorphism types of regular embedded blowups of $U$ along $Z.$ 
\end{itemize} 
\end{thm}

\begin{proof} (a): By Lemma~\ref{Lemma EI1} there is a strongly regular pair $\underline{f} \in \mathbb{R}[{\bf x},{\bf y}]^2$ with respect to $Z$  
such that $\mathrm{det}\big(\partial \underline{f}(p)\big) = \sigma(p)$ for all $p \in Z.$ It suffices to choose $B = \mathrm{Bl}_U(\underline{f}).$

(b): We may write $B = \mathrm{Bl}_U(\underline{g}),$ where $\underline{g} \in \mathbb{R}[{\bf x},{\bf y}]$ is a regular pair of polynomials with respect to $Z$ on $U.$ \\
Assume first that $B$ and $C$ are oriented embedded isomorphic, more precisely, that $C = \varphi_M(B)$ for some automorphism $\varphi_M: U \times \mathbb{P}^1 \longrightarrow  U \times \mathbb{P}^1$ with $M \in \mathbb{R}[{\bf x},{\bf y}]^{2 \times 2}$ and $\mathrm{det}\big(M(p)\big) > 0$ for all $p \in U.$ Then we may write $C= \mathrm{Bl}_U(\underline{g}M).$ By the product rule for derivatives (see (@@), Proof of Lemma~\ref{Proposition EI1}), as $\underline{g}(Z) = 0$ and as $\mathrm{det}(M(p)) > 0$ for all $p \in U,$ we now obtain 
\begin{align*}
\mathrm{sgn}_C(p) &= \mathrm{sgn}\big(\mathrm{det}[\partial(\underline{g}M)(p)]\big) = \mathrm{sgn}\big(\mathrm{det}[(\partial\underline{g})(p)M(p)]\big) = \\ 
                  &= \mathrm{sgn}\big(\mathrm{det}[\partial\underline{g}(p)]\mathrm{det}[M(p)]\big) = \mathrm{sgn}\big(\mathrm{det}[\partial\underline{g}(p)]\big) = \\
                  &= \mathrm{sgn}_B(p) \mbox{ {\rm for all} } p \in Z.
\end{align*} 
 It follows that indeed $\mathrm{sgn}_C = \mathrm{sgn}_B.$ \\
Assume conversely, that $\mathrm{sgn}_C = \mathrm{sgn}_B.$ By Lemma~\ref{Lemma EI1} there is a strongly regular pair $\underline{f} \in \mathbb{R}[{\bf x},{\bf y}]^2$ with respect to $Z$ on $U$ such that $\mathrm{det}\big(\partial \underline{f}(p)\big) =  \mathrm{sgn}_B(p) = \mathrm{sgn}_C(p) \mbox{ {\rm for all} } p\in Z.$\\ 
By Lemma~\ref{Proposition EI1} there is a matrix $M \in \mathbb{R}[{\bf x},{\bf y}]^{2 \times 2}$ such that $\mathrm{det}\big(M(p)\big) > 0$ for all $p \in U$ and $\underline{g} = \underline{f}M.$  But this means, that $D := \mathrm{Bl}_U(\underline{f}) \cong B$. Similarly we see, that $D \cong C$. So $B$ and $C$ are embedded isomorphic.

(c):  This is clear by statements (a) and (b).
\end{proof}

\begin{rem}\label{Problem EI1} {\rm  A (more complicated) proof of the 
Classification Theorem~\ref{Theorem EI1} -- based on the ideas of the first named author -- has been worked out in the Master thesis of S. Koller \cite{Kol}, but remained unpublished yet. }
\end{rem}

\section{Deformation of Matrices and Isotopies of Embedded Blowups}

\noindent {\rm \bf{Analytic Matrix Deformations.}}  In this section, we approach the Deformation Problem (1.0)(a) mentioned in the introduction. We shall prove the Deformation Result (1.13). As already mentioned in the introduction, this means that we have to prove the result on polynomial deformations of matrices mentioned in (1.14).
We first prove a result on real analytic deformation of matrices.

\begin{nota rem}\label{Notation Remark CF1} {\rm (A) Let $\mathcal{C}^\omega(U)$ denote the ring of real analytic 
functions on $U$. We choose a matrix}  
$$M = (M_{\bullet 1} \; M_{\bullet 2}) = \begin{pmatrix}M_{11} & M_{12}\\M_{21} & M_{22}\end{pmatrix} \in 
\mathcal{C}^\omega(U)^{2 \times 2} \mbox{ {\rm with} } \mathrm{det}\big(M(p)\big) > 0 \mbox{ {\rm for all} } p 
\in U,$$
{\rm where} 
$$M_{\bullet 1} := \begin{pmatrix}M_{11} \\ M_{21}\end{pmatrix} \mbox{ {\rm and} } M_{\bullet 2} := 
\begin{pmatrix}M_{12} \\ M_{22}\end{pmatrix}$$
{\rm denote the column vectors of $M.$}

{\rm   Let $p,q \in U$ and let $\sigma:[0,1] \longrightarrow U$ be a smooth path with $\sigma(0) = p$ 
and $\sigma(1) = q.$ As $U$ is pathwise simply connected -- by monodromy -- the two values below (which are the total angles the vectors $\frac{M_{\bullet 1}}{\|M_{\bullet 1}\|}\big(\sigma(t)\big)$ and 
$\frac{M_{\bullet 2}}{\|M_{\bullet 2}\|}\big(\sigma(t)\big)$ respectively wander through if $t$ runs from $0$ to $1$)}
\begin{align*}
\alpha_M(p,q) &= \int_0^1 \frac{M_{\bullet 1}}{\|M_{\bullet 1}\|}\big(\sigma(t)\big) \wedge \frac{d}{dt}\big[\frac{M_{\bullet 1}}{\|M_{\bullet 1}\|}\big(\sigma(t)\big)\big] dt, \\
\beta_M(p,q) &= \int_0^1 \frac{M_{\bullet 2}}{\|M_{\bullet 2}\|}\big(\sigma(t)\big) \wedge \frac{d}{dt}\big[\frac{M_{\bullet 2}}{\|M_{\bullet 2}\|}\big(\sigma(t)\big)\big] dt
\end{align*}
{\rm  depend only (analytically) on $p$ and $q$ and not on their connecting path $\sigma.$ Now, we fix a point $p_0 \in U.$ Then, there are uniquely determined functions
$\alpha_M, \beta_M: U \longrightarrow \mathbb{R}$ such that }
$$ 0 \leq \alpha_M(p_0),  \beta_M(p_0) \leq 2\pi,$$
$$ \alpha_M(p) =  \alpha_M(p_0) + \alpha_M(p_0,p) \mbox{ {\rm and} } \beta_M(p) =  \beta_M(p_0) + \beta_M(p_0,p) 
\mbox{ {\rm for all} } p \in U$$  
{\rm and}
$$ M_{\bullet 1}(p) = \|M_{\bullet 1}(p)\|\begin{pmatrix}\mathrm{cos}(\alpha_M(p))\\ \mathrm{sin}(\alpha_M(p))\end{pmatrix}, \;
M_{\bullet 2}(p) = \|M_{\bullet 2}(p)\|\begin{pmatrix}\mathrm{cos}(\beta_M(p))\\ \mathrm{sin}(\beta_M(p)) \end{pmatrix},\mbox{ {\rm for all} } 
p \in U.$$
{\rm  Observe, that in particular}
$$
\mathrm{det}\big(M(p)\big) = \|M_{\bullet 1}(p)\|\cdot\|M_{\bullet 2}(p)\|\cdot\mathrm{sin}\big(\beta_M(p)-\alpha_M(p)\big) > 0 \mbox{ {\rm for all} }p \in U.$$
{\rm Now, by continuity, and as $\alpha_M(p,q)$ and $\beta_M(p,q)$ depend analytically on $p$ and $q,$ it follows that}
\begin{itemize}
\item[\rm{(a)}] {\rm $ 0 < \beta_M(p)-\alpha_M(p) < \pi$ for all  $p \in U$;}
\item[\rm{(b)}] {\rm $\alpha_M, \beta_M \in \mathcal{C}^\omega(U)$.}
\end{itemize}

{\rm (B) Keep the notations and hypotheses of part (A). For each $t \in [0,1]$ and each $p \in U$ we set}
\begin{align*}
M^{(t)}_{11}(p) &:= \big[(1-t) + t\|M_{\bullet 1}\|\big] \cdot \mathrm{cos}\big(t\alpha_M(p)\big),\\
M^{(t)}_{21}(p) &:= \big[(1-t) + t\|M_{\bullet 1}\|\big] \cdot \mathrm{sin}\big(t\alpha_M(p)\big),\\
M^{(t)}_{12}(p) &:= \big[(1-t) + t\|M_{\bullet 2}\|\big] \cdot \mathrm{cos}\big((1-t)\frac{\pi}{2} + t\beta_M(p)\big),\\
M^{(t)}_{22}(p) &:= \big[(1-t) + t\|M_{\bullet 2}\|\big] \cdot \mathrm{sin}\big((1-t)\frac{\pi}{2} + t\beta_M(p)\big),\\
\end{align*}
{\rm and consider the matrices}
$$M^{(t)} := \begin{pmatrix}M^{(t)}_{11} & M^{(t)}_{12}\\ M^{(t)}_{21} & M^{(t)}_{22} \end{pmatrix} \in 
\mathcal{C}(U)^{2\times 2}, \; \big(t \in [0,1]\big).$$
{\rm For all $t \in [0,1]$ and all $p \in U$ we obtain:}
\begin{align*}
&\mathrm{det}\big(M^{(t)}(p)\big) =\\
&= \big[(1-t) + t\|M_{\bullet 1}(p)\|\big]\cdot\big[(1-t) + t\|M_{\bullet 2}(p)\|\big]\cdot\mathrm{sin}\big((1-t)\frac{\pi}{2} + t[\beta_M(p) - \alpha_M(p)]\big).
\end{align*}
{\rm Moreover, $0 < \beta_M(p) - \alpha_M(p) < \pi$ (see statement (a) of Part (A)) implies} 
$$0 < (1-t)\frac{\pi}{2} + t[\beta_M(p) - \alpha_M(p)] < (1-t)\frac{\pi}{2} + t\pi = \frac{\pi}{2} + t\frac{\pi}{2} \leq \pi.$$ 
{\rm So, in view of statement (b) of part (A) we can say:}
 
\begin{itemize}
\item[\rm{(a)}] {\rm $M^{(t)} \in \mathcal{C}^\omega(U)^{2 \times 2} \mbox{ {\rm and} } \mathrm{det}\big(M^{(t)}(p)\big) > 0 
\mbox{ {\rm for all} } t \in [0,1] \mbox{ {\rm and all} } p \in U .$}
\end{itemize}
 
\end{nota rem}

{\rm Now, we solve our deformation problem for matrices with analytic entries.}

\begin{prop}\label{Proposition CF 1} Let $M \in \mathcal{C}^\omega(U)^{2 \times 2}$ such that $\mathrm{det}\big(M(p)\big) > 0$ for all $p \in U.$ Then the family $\big({M}^{(t)}\big)_{0\leq t \leq 1}$ of Notation and Remark~\ref{Notation Remark CF1} is an analytic family of matrices in $\mathcal{C}^\omega(U)^{2 \times 2},$ with positive determinant on $U,$ which connects the unit matrix ${\bf 1}^{2\times 2}$ with the matrix $M.$ More precisely,
\begin{itemize}
\item[\rm{(a)}] $M^{(t)} \in \mathcal{C}^\omega(U)^{2 \times 2} \mbox{ {\rm and} } \mathrm{det}\big(M^{(t)}(p)\big) > 0 \mbox{ {\rm for all} } t \in [0,1] \mbox{ {\rm and all} } p \in U .$
\item[\rm{(b)}] $M^{(0)} = {\bf 1}^{2\times2}$ and $M^{(1)} = M.$
\item[\rm{(c)}] The map $\widetilde{M}:U\times[0,1] \longrightarrow \mathbb{R}^{2\times2},$ given by $(p,t) \mapsto M^{(t)}(p),$ is continuous and analytic on the open set $U \times ]0,1[.$
\end{itemize} 
\end{prop}
\begin{proof} (a): This is immediate by Notation and Remark~\ref{Notation Remark CF1} (B)(a).

(b): This is obvious by the definition of the Matrices $M^{(t)}$.

(c): This follows easily from the definition of the functions $p \mapsto M^{(t)}_{ij}(p)$ (see Notation and Remark~\ref{Notation Remark CF1} (B)) and statement
(b) of Notation and Remark~\ref{Notation Remark CF1} (A).
\end{proof}

\noindent {\rm \bf{Polynomial and Rational Matrix Deformations.}} {\rm We now attack the case of polynomial or rational matrix deformations. We begin with the following auxiliary result.}

\begin{lem}\label{Lemma RCF 1} Let $K \subset \mathbb{R}^2$ be a non-empty compact set. Let $P,Q \in \mathbb{R}[{\bf x}, {\bf y}]$ be two polynomials and let
$F: K \times [0,1] \longrightarrow \mathbb{R}$ be a continuous function such that $F(p,0) = P(p)$ and $F(p,1) = Q(p)$ for all $p \in K$. Let $\varepsilon > 0.$
Then, there is a polynomial $\widetilde{P} \in \mathbb{R}[{\bf x},{\bf y}, {\bf t}]$ such that
\begin{itemize}
\item[\rm{(a)}] $|F(p,t) - \widetilde{P}(p,t)| < \varepsilon$ for all $p \in K$ and all $t \in [0,1].$
\item[\rm{(b)}] $P(p) = \widetilde{P}(p,0)$ and $Q(p) = \widetilde{P}(p,1)$ for all $p \in K.$ 
\end{itemize} 
\end{lem}
\begin{proof} {\rm By the Theorem of Stone-Weierstrass (see \cite{D} (7.4.1)) there is a polynomial $\overline{P} \in \mathbb{R}[{\bf x},{\bf y}, {\bf t}]$ such that}
$$|F(p,t) - \overline{P}(p,t)| < \frac{\varepsilon}{2} \mbox{ {\rm for all} } p \in K \mbox{ {\rm and all} } t \in [0,1].$$
{\rm Now, set}
$$\widetilde{P}({\bf x},{\bf y}, {\bf t}) := 
\overline{P}({\bf x},{\bf y},{\bf t}) + (1- {\bf t})\big(P({\bf x},{\bf y}) - \overline{P}({\bf x},{\bf y},0)\big) + {\bf t}\big(Q({\bf x},{\bf y}) - \overline{P}({\bf x},{\bf y} ,1)\big).$$
{\rm It is easy to see that $\widetilde{P}$ has the requested properties.}
\end{proof}

\begin{prop}\label{Proposition RCF 1} Let $M,N \in \mathbb{R}[{\bf x},{\bf y}]^{2\times 2}$ such that $\mathrm{det}\big(M(p)\big) > 0$ and $\mathrm{det}\big(N(p)\big) > 0$ for all $p \in \overline{U}.$ 
Then, the matrix $N$ is connected on $\overline{U}$ to $M$ by a polynomial family of polynomial $2\times 2$-matrices with positive determinant on $\overline{U}.$ More precisely:\\
There is a matrix
$$\widetilde{P} = \begin{pmatrix}\widetilde{P}_{11} & \widetilde{P}_{12} \\ \widetilde{P}_{21} & \widetilde{P}_{22}\end{pmatrix} \in \mathbb{R}[{\bf x},{\bf y},{\bf t}]^{2\times 2}$$
such that with $P^{(t)}({\bf x},{\bf y}) := \widetilde{P}({\bf x},{\bf y},t)$ (for $t \in \mathbb{R}$) we have:  
\begin{itemize}
 \item[\rm{(a)}] $P^{(0)}(p) = N(p)$ for all $p \in \overline{U}.$
 \item[\rm{(b)}] $P^{(1)}(p) = M(p)$ for all $p \in \overline{U}.$
 \item[\rm{(c)}] $\mathrm{det}\big(P^{(t)}(p)\big) > 0$ for all $p \in \overline{U}$ and all $t \in [0,1].$
\end{itemize}
\end{prop}

\begin{proof} {\rm Observe that the closed set } 
$$\mathbb{S} := \{p \in \mathbb{R}^2 \mid \mathrm{det}\big(M(p)\big) \leq 0\ \mbox{ {\rm or} } \mathrm{det}\big(N(p)\big) \leq 0\}$$ 
{\rm is disjoint to $\overline{U}.$ We thus find a bounded open star-shaped set $W$ such that $\overline{U} \subset W$ and $W \cap \mathbb{S} = \emptyset.$ Now, clearly $M,N \in \mathbb{R}[{\bf x},{\bf y}]^{2 \times 2}$ with $\mathrm{det}\big(M(p)\big), \mathrm{det}\big(N(p)\big) > 0$ for all $p \in W.$ According to Proposition~\ref{Proposition CF 1} we have two continuous maps}
\begin{align*}
\widetilde{M} &= \begin{pmatrix}\widetilde{M}_{11} & \widetilde{M}_{12} \\\widetilde{M}_{21} & \widetilde{M}_{22}\end{pmatrix}: W\times [0,1] \longrightarrow \mathbb{R}^{2\times 2} \mbox{ {\rm with} } \mathrm{det}\big(\widetilde{M}(p,t)\big) > 0, \mbox{ {\rm for all} } (p,t) \in W\times[0,1],\\
\widetilde{N} &= \begin{pmatrix}\widetilde{N}_{11} & \widetilde{N}_{12} \\\widetilde{N}_{21} & \widetilde{N}_{22}\end{pmatrix}: W\times [0,1] \longrightarrow \mathbb{R}^{2\times 2} \mbox{ {\rm with} } \mathrm{det}\big(\widetilde{N}(p,t)\big) > 0, \mbox{ {\rm for all} } (p,t) \in W\times[0,1],
\end{align*}
{\rm such that}
\begin{align*}
\widetilde{M}(p,0) &= {\bf 1}^{2 \times2},\mbox{ {\rm and} }  \widetilde{M}(p,1) = M(p), \mbox{ {\rm for all} } p \in W,\\
\widetilde{N}(p,0) &= {\bf 1}^{2 \times2},\mbox{ {\rm and} }  \widetilde{N}(p,1) = N(p), \mbox{ {\rm for all} } p \in W.
\end{align*}
{\rm Now, for all $i,j \in \{1,2\}$ we consider the continuous functions}
$$\widetilde{F}_{ij}:W \times [0,1] \longrightarrow \mathbb{R} ; \quad \widetilde{F}_{i,j}(p,t) := \begin{cases}\widetilde{N}_{ij}(p,1-2t) & \mbox{ {\rm if } } t \in [0,\frac{1}{2}]\\
                                                                                  \widetilde{M}_{ij}(p,2t-1) & \mbox{ {\rm if } } t \in [\frac{1}{2},1]            
                                                                                 \end{cases}$$
{\rm and the matrix (here $\mathcal{C}(\bullet)$ denotes the ring of continuous functions)}
$$ \widetilde{F} := \begin{pmatrix}\widetilde{F}_{11} & \widetilde{F}_{12} \\ \widetilde{F}_{21} & \widetilde{F}_{22}\end{pmatrix} \in \mathcal{C}(W \times [0,1])^{2\times 2}.$$ 
{\rm Then $\widetilde{F}(p,0) = N(p), \widetilde{F}(p,1) = M(p)$ and  $\mathrm{det}\big(\widetilde{F}(p,t)\big) > 0$ for all $p \in W$ and all $t \in [0,1].$ 
As $\overline{U} \subset W$ is compact, there are $c, \delta > 0$ such that for all $i,j \in \{1,2\},$ all $p \in \overline{U}$ and all $t \in [0,1]$ it holds}
$$-c \leq \widetilde{F}_{ij}(p,t) \leq c \; \mbox{ {\rm and} } \;\mathrm{det}\big(\widetilde{F}(p,t)\big) > \delta.$$
{\rm As the map $\mathrm{det}:\mathbb{R}^{2\times2}\longrightarrow \mathbb{R}$ is uniformly continuous on any compact subset of $\mathbb{R}^4$ we find some $\varepsilon > 0$ such that:}
\begin{itemize}
\item[\rm{(1)}] $|\mathrm{det}\big(\widetilde{F}(p,t)\big) - \mathrm{det}\begin{pmatrix}m_{11} & m_{12}\\m_{21} & m_{22}\end{pmatrix}| < \frac{\delta}{2}$ {\rm for all
$p \in \overline{U}$, all $t \in [0,1]$ and all $m_{ij} \in \mathbb{R}$ with $|m_{ij}-\widetilde{F}_{ij}(p,t)| < \varepsilon  \quad (i,j \in \{1,2\})$.}
\end{itemize}
{\rm Now, we apply Lemma~\ref{Lemma RCF 1} to the four continuous functions $\widetilde{F}_{ij}:\overline{U}\times [0,1] \longrightarrow \mathbb{R}$ and obtain four polynomials
$\widetilde{P}_{ij} \in \mathbb{R}[{\bf x},{\bf y},{\bf t}],$ such that for all $i,j \in \{1,2\}$ we have:} 
\begin{itemize}
\item[\rm{(2)}] $|\widetilde{F}_{ij}(p,t) - \widetilde{P}_{ij}(p,t)| < \varepsilon$ for all $p \in \overline{U}$ and all $t \in [0,1],$
\item[\rm{(3)}] $N_{ij}(p) = \widetilde{F}_{ij}(p,0) = \widetilde{P}_{ij}(p,0)$ for all $p \in \overline{U}$ and 
\item[\rm{(4)}] $M_{ij}(p) = \widetilde{F}_{ij}(p,1) = \widetilde{P}_{ij}(p,1)$ for all $p \in \overline{U}.$ 
\end{itemize}
We set 
$$
\widetilde{P} := \begin{pmatrix}\widetilde{P}_{11} & \widetilde{P}_{12} \\ \widetilde{P}_{21} & \widetilde{P}_{22}\end{pmatrix}.
$$
 Then, the above statements (1) and (2) yield that
$$
|\mathrm{det}\big(\widetilde{F}(p,t)\big) - \mathrm{det}\big(\widetilde{P}(p,t)\big)| < \frac{\delta}{2} \mbox{ 
	 for all } p \in \overline{U} \mbox{ {\rm and all} } t \in [0,1],
$$
so that with $P^{(t)}(p) := \widetilde{P}(p,t)$ and  (because also $\det(\tilde{F}(p,t))>\delta$) we get
$$\mathrm{det}\big(P^{(t)}(p)\big) = \mathrm{det}\big(\widetilde{P}(p,t)\big) > \frac{\delta}{2} >0 \mbox{ {\rm for all} } p \in \overline{U} \mbox{ {\rm and all} } t \in [0,1].$$ 
{\rm By the above statements (3) and (4) we obtain}
$$P^{(0)}(p) =\widetilde{P}(p,0) = N(p) \mbox{ {\rm and} } P^{(1)}(p) = \widetilde{P}(p,1) = M(p) \mbox{ {\rm for all} } p \in \overline{U}.$$
{\rm Altogether, this proves our claim.} 
\end{proof}

\begin{rem}\label{Remark RCF 1} {\rm As an immediate consequence we now get the result announced in the introduction under (1.14).} 
\end{rem}

\begin{rem}\label{Remark RCF} {\rm As early as 2002, the first named author did ask for the existence of a connecting family $(M^{(t)})_{t \in [0,1]}$ as in  Proposition~\ref{Proposition RCF 1} -- but only continuous, not polynomial -- at the occasion of a talk he gave at the IIT Bombay. A few weeks after this, A.R. Shastri \cite{Sh} suggested a proof for the existence of a piecewise linear connecting family $(M^{(t)})_{t \in [0,1]}.$ The authors are grateful to him for his hint. Clearly, instead of Proposition~\ref{Proposition CF 1} one also could use Shastri's result to prove Proposition~\ref{Proposition RCF 1}.} 
\end{rem}

{\rm As an easy consequence of the above proposition we now get:} 

\begin{cor}\label{Corollary RCF 2} Let $M = \begin{pmatrix} M_{11} & M_{12} \\ M_{21} & M_{22} \end{pmatrix} \in \mathbb{R}({\bf x},{\bf y})^{2 \times 2}$ be such that none of its entries $M_{ij}, (i,j \in \{1,2\})$ has a pole in $\overline{U},$  and such that $\mathrm{det}\big(M(p)\big) > 0$ for all $p \in \overline{U}.$\\
Then, the unit matrix ${\bf 1}^{2\times 2}$ is connected over $\overline{U}$ to $M$ by a rational family of $2\times 2$-matrices which are defined and of positive determinant on $\overline{U}$. More precisely:\\
There is a matrix
$$\widetilde{Q} = \begin{pmatrix}\widetilde{Q}_{11} & \widetilde{Q}_{12} \\ \widetilde{Q}_{21} & \widetilde{Q}_{22}\end{pmatrix} \in \mathbb{R}({\bf x},{\bf y},{\bf t})^{2\times 2}$$
such that no $\widetilde{Q}_{ij}$ has a pole on $\overline{U}$ and such that, with $Q^{(t)}({\bf x},{\bf y}) : = \widetilde{Q}({\bf x},{\bf y},t)$ (for $t \in \mathbb{R}$): 
\begin{itemize}
\item[\rm{(a)}] $Q^{(0)} = {\bf 1}^{2 \times 2}.$
\item[\rm{(b)}] $Q^{(1)}(p) = M(p)$ for all $p \in \overline{U}.$
\item[\rm{(c)}] $\mathrm{det}\big(Q^{(t)}(p)\big) > 0$ for all $p \in \overline{U}$ and all $t \in [0,1].$
\end{itemize}
\end{cor}

\begin{proof} {\rm The closed set 
$$\mathcal{P} := \bigcup_{1 \leq i,j \leq 2} \mathrm{Pole}(M_{ij}) \cup \{p \in \mathbb{R}^2 \mid \mathrm{det}\big(M(p)\big) \leq 0\}$$ 
is disjoint to $\overline{U}.$ We thus find a bounded open star-shaped set $W$ such that $\overline{U} \subset W$ and $W \cap \mathcal{P} = \emptyset.$ So, none of the four entries $M_{ij}$ of $M$ has a pole in $W$ and moreover $\mathrm{det}\big(M(p)\big) > 0$ for all $p \in W.$ As $W$ is path-wise connected and by taking common denominators we find} 
$$H \in \mathbb{R}[{\bf x},{\bf y}]^{2\times2} \mbox{ {\rm and } } G \in \mathbb{R}[{\bf x},{\bf y}] \mbox{ {\rm with} } G(p) > 0 \mbox{ {\rm and} } M(p) = \frac{H(p)}{G(p)} \mbox{ {\rm for all} } p \in W.$$
{\rm In particular we have $\mathrm{det}\big(G(p){\bf 1}^{2\times 2}\big) > 0$  and $\mathrm{det}\big(H(p)\big) > 0$ for all $p \in W,$ hence for all $p \in \overline{U}.$
By Proposition~\ref{Proposition RCF 1} there is a matrix $\widetilde{P} \in \mathbb{R}[{\bf x},{\bf y},{\bf t}]^{2\times 2}$
such that}
\begin{itemize}
 \item[\rm{(1)}] $\widetilde{P}(p,0) = G(p){\bf 1}^{2\times2}$ for all $p \in \overline{U};$
 \item[\rm{(2)}] $\widetilde{P}(p,1) = H(p)$ for all $p \in \overline{U};$
 \item[\rm{(3)}] $\mathrm{det}\big(\widetilde{P}(p,t)\big) > 0$ for all $p \in \overline{U}$ and all $t \in [0,1].$
\end{itemize}
{\rm Now, with $\widetilde{Q}: = \frac{\widetilde{P}}{G}$ we get our claim.}
\end{proof}

\noindent {\rm \bf{Isotopies of Embedded Blowups.}} {\rm As an application of Proposition~\ref{Proposition RCF 1} we now prove the result on the deformation of regular embedded blowups by means of isotopies mentioned in (1.13). }

\begin{thm}\label{Theorem REIS}  Let $B, C \in \mathfrak{Bl}_U(Z)$ be such that $B$ and $C$ are 
relatively oriented embedded isomorphic. 
Then, $B$ and $C$ are connected by an isotopy of $U \times \mathbb{P}^1$-automorphisms. More precisely, there is a matrix
$$\widetilde{M} = \begin{pmatrix}\widetilde{M}_{11} & \widetilde{M}_{12} \\ \widetilde{M}_{21} & \widetilde{M}_{22}\end{pmatrix} \in \mathbb{R}[{\bf x},{\bf y},{\bf t}]^{2\times 2}$$
such that with $M^{(t)}({\bf x},{\bf y}) := \widetilde{M}({\bf x},{\bf y},t) \big(\mbox{ for } t \in \mathbb{R})$ we have:   
\begin{itemize}
 \item[\rm{(a)}] $\mathrm{det}\big(M^{(t)}(p)\big) > 0$ for all $p \in U$ and all $t \in [0,1]$  -- and hence $\varphi^{(t)} := \varphi_{M^{(t)}}$ is a relative oriented automorphism of $U \times \mathbb{P}^1$ for all $t \in [0,1]$.
 \item[\rm{(b)}] $\varphi^{(0)}(B) = B$ and $\varphi^{(1)}(B) = C.$
 \end{itemize}
\end{thm}

\begin{proof} {\rm Let $\underline{f} \in \mathbb{R}[{\bf x},{\bf y}]^2$ be such that $Z_{\overline{U}}(\underline{f}) = Z.$ As $B \cong C$, we find some matrix $N \in \mathbb{R}[{\bf x},{\bf y}]^{2 \times 2}$ with $\mathrm{det}\big(N(p)\big) > 0$ for all $p \in U$ and such that, with $(g_0,g_1) = \underline{g} := \underline{f}N,$ it holds $C = \mathrm{Bl}_U(\underline{g})$ (see (1.7)). Now, we choose $\gamma \in \mathbb{R}_{>0}$ and consider the matrix 
$$M := N_\gamma = N + \gamma\begin{pmatrix}g_1 f_1 & -g_0 f_1\\ -g_1 f_0 & g_0 f_0\end{pmatrix}$$
of Lemma~\ref{Lemma EI2}. Then, by statements (b), (c) and (d) of that Lemma and as $g_0$ and $g_1$ have no common zero on the boundary of $U,$ it follows that for $\gamma$ large enough we have $\mathrm{det}\big(M(p)) > 0$ for all $p \in \overline{U}$ and $\underline{g} = \underline{f}M.$ \\
But now Proposition~\ref{Proposition RCF 1} yields that there is a matrix $\widetilde{M} \in \mathbb{R}[{\bf x},{\bf y},{\bf t}]^{2\times 2}$ such that, with $M^{(t)}({\bf x},{\bf y}) := \widetilde{M}({\bf x},{\bf y},t),$ 
it holds }
\begin{itemize}
 \item[\rm{(1)}] {\rm $M^{(0)}(p) = {\bf 1}^{2\times 2}$ for all $p \in \overline{U};$}
 \item[\rm{(2)}] {\rm $M^{(1)}(p) = M(p)$ for all $p \in \overline{U};$}
 \item[\rm{(3)}] {\rm $\mathrm{det}\big(M^{(t)}(p)\big) > 0$ for all $p \in \overline{U}$ and all $t \in [0,1].$}
\end{itemize}
{\rm In particular, we get the stated existence of the matrix $\widetilde{M} \in \mathbb{R}[{\bf x},{\bf y},{\bf t}]^{2\times 2}$ such that statement (a) holds. \\
As $\varphi^{(0)}(B) = \varphi_{M^{(0)}}(B) = \varphi_{{\bf 1}^{2 \times 2}}(B)  = \mathrm{id}_{U \times \mathbb{P}^1}(B) = B$ and}
$C = \mathrm{Bl}_U(\underline{f}M)= \mathrm{Bl}_U(\underline{f}M^{(1)}) = \varphi_{M^{(1)}}\big(\mathrm{Bl}_U(\underline{f})\big) = \varphi_{M^{(1)}}(B) = \varphi^{(1)}(B)$, 
{\rm we get statement (b).}

\end{proof}

\section{Further Examples of Families of Blowups}

\noindent {\rm \bf{Two Families of Regular Two-point Blowups.}} {\rm Already in Example~\ref{Exam.Moebius} and Example~\ref{Exam.4Points} 
we have presented deformations of regular blowups by means of a particularly simple matrix deformation.  We begin the present section with slightly more involved matrix 
deformations and we shall illustrate their effect on two non-isomorphic regular embedded two-point blowups.  We fix our settings as in the examples given in the introduction and in Section 2 by 
choosing $\rho =2, r = 4, U = \{(x,y) \in \mathbb{R}^2 \mid x^2 + y^2 < 4\}.$}

\begin{exas}\label{Example const det} {\rm (A) We fix a polynomial $a = a({\bf x},{\bf y}) \in \mathbb{R}[{\bf x},{\bf y}]$ and consider the matrix}
$$\widetilde{M} = \widetilde{M}({\bf x},{\bf y},{\bf t}) := \begin{pmatrix} 1 - a({\bf x},{\bf y}){\bf t} & a({\bf x},{\bf y}){\bf t} \\ -a({\bf x},{\bf y}){\bf t} & 1 + a({\bf x},{\bf y}){\bf t}\end{pmatrix} \in \mathbb{R}[{\bf x},{\bf y},{\bf t}]^{2 \times 2} \mbox{  \rm{with} } \mathrm{det}(\widetilde{M}) = 1$$
{\rm and the matrices} 
$$M^{(t)} = M^{(t)}({\bf x},{\bf y}) := \widetilde{M}({\bf x},{\bf y},t) \in \mathbb{R}[{\bf x},{\bf y}]^{2 \times 2} 
\mbox{ {\rm with} } \mathrm{det}(M^{(t)}) = 1 \mbox{ {\rm for all} } t \in \mathbb{R}.$$
{\rm So, for any regular blowup $B = \mathrm{Bl}_U(\underline{f}) = \mathrm{Bl}_U(f_0,f_1) \in \mathfrak{Bl}_U^{\mathrm{reg}}(Z)$ we get an isotopic family} 
$$\big(B^{(t)} = \mathrm{Bl}_U(\underline{f}M^{(t)})\big)_{t \in [0,1]} \mbox{ {\rm such that for all} } t \in [0,1] \mbox{ {\rm it holds:}}$$ 
$$B^{(t)} = \mathrm{Bl}_U\big(f_0 - t\cdot a({\bf x},{\bf y})(f_0 + f_1),f_1 + t\cdot a({\bf x},{\bf y})(f_0 + f_1) \big) \in 
\mathfrak{Bl}_U^{\mathrm{reg}}(Z) \mbox{ {\rm and} } B^{(t)} \cong B.$$ 
{\rm We thus get a family $\big(B^{(t)}\big)_{t \in [0,1]}$ of isotopic blowups $B^{(t)} \in \mathfrak{Bl}_U^{\mathrm{reg}}(Z),$ which connects $B = B^{(0)}$ with} 
$$C := B^{(1)} = \mathrm{Bl}_U(\underline{f}M^{(1)}) = \mathrm{Bl}_U\big(f_0 - a({\bf x},{\bf y})(f_0+f_1),f_1 + a({\bf x},{\bf y})(f_0+f_1)\big).$$
{\rm As announced, we aim to illustrate the situation by means of two regular two-point blowups, which are of 
	are of different (relative oriented embedded) isomorphism type, 
a situation which can indeed only occur for regular blowups with respect to more than one point. More precisely, we shall blow up $U$ with respect to 
two different pairs $\underline{f}$ of regular polynomials which both satisfy $Z_U(\underline{f}) = \{(\pm 1,0)\},$ 
but such that the sign distribution $\mathrm{sgn}_{\underline{f}}$ (see Definition and Remark ~\ref{Definition sgn}) is non-constant in the first case and constant in the second case.}

{\rm  (B) We keep the general settings of part (A), set $a({\bf x},{\bf y}) := {\bf x}{\bf y}$ and consider the regular two-point blowup $B:= \mathrm{Bl}_U(\underline{f})$ of $U$ with respect to $Z := \{(\pm 1,0)\}$ given by $f_0 := {\bf x}^2 + {\bf y}^2 -1$ and $f_1 := {\bf y}.$ We then have $\mathrm{sgn}_B\big((\pm 1,0)\big) = \pm 1$, so that the sign distribution $\mathrm{sgn}_B = \mathrm{sgn}_{\underline{f}}$ is non-constant. The visualization of the resulting family of two-point blowups $B^{(t)} \cong B^{(0)} = B$ is presented in Figure 6 for $t = 0,  0.5, 1.$}

\begin{figure}
	\centering
	\begin{tabular}{ccc}
		\includegraphics[width=.3\linewidth]{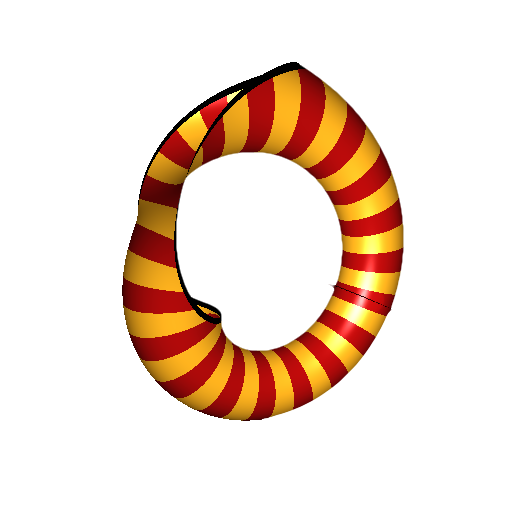} &  
		\includegraphics[width=.3\linewidth]{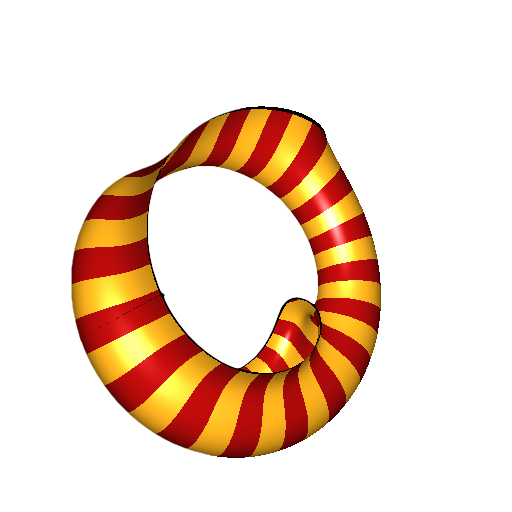} &
		\includegraphics[width=.3\linewidth]{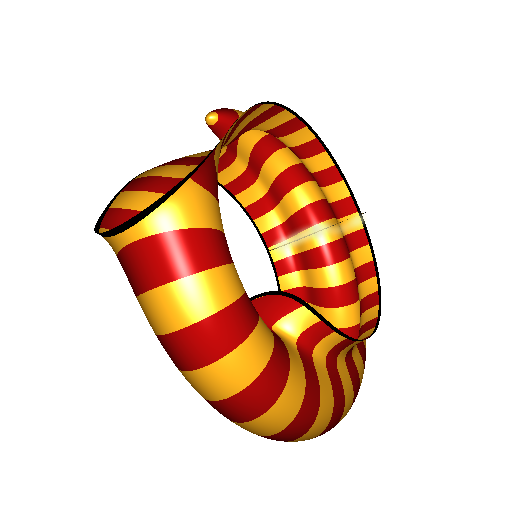} 
		\\
		\small $B^{(0)} = B$ & \small $B^{(0.5)}$ & \small $B^{(1)}=C$
	\end{tabular}
\caption{Deformation of a regular two-point blowup with non-constant sign distribution}

\end{figure}

{\rm (C) We now choose $a({\bf x},{\bf y}) := {\bf y}$ and consider the the regular two-point blowup $B:= \mathrm{Bl}_U(\underline{f})$ of $U$ with respect to $Z := \{(\pm 1,0)\}$
given by $f_0 := {\bf x}^2 -1$  and $f_1 := {\bf x}{\bf y}.$ This time, it holds $\mathrm{sgn}_B\big((\pm 1,0)\big) = 1$, so that the sign distribution $\mathrm{sgn}_B = \mathrm{sgn}_{\underline{f}}$ is constant. This means, that we get a two-point blowup whose embedded isomorphism type differs  from the isomorphism type of the blowup of part (B).  The visualization of the resulting family of two-point blowups $B^{(t)} \cong B^{(0)} = B$ is presented in Figure 7 for $t = 0, 0.5, 1.$}

\begin{figure}
	\centering
	\begin{tabular}{ccc}
		\includegraphics[width=.3\linewidth]{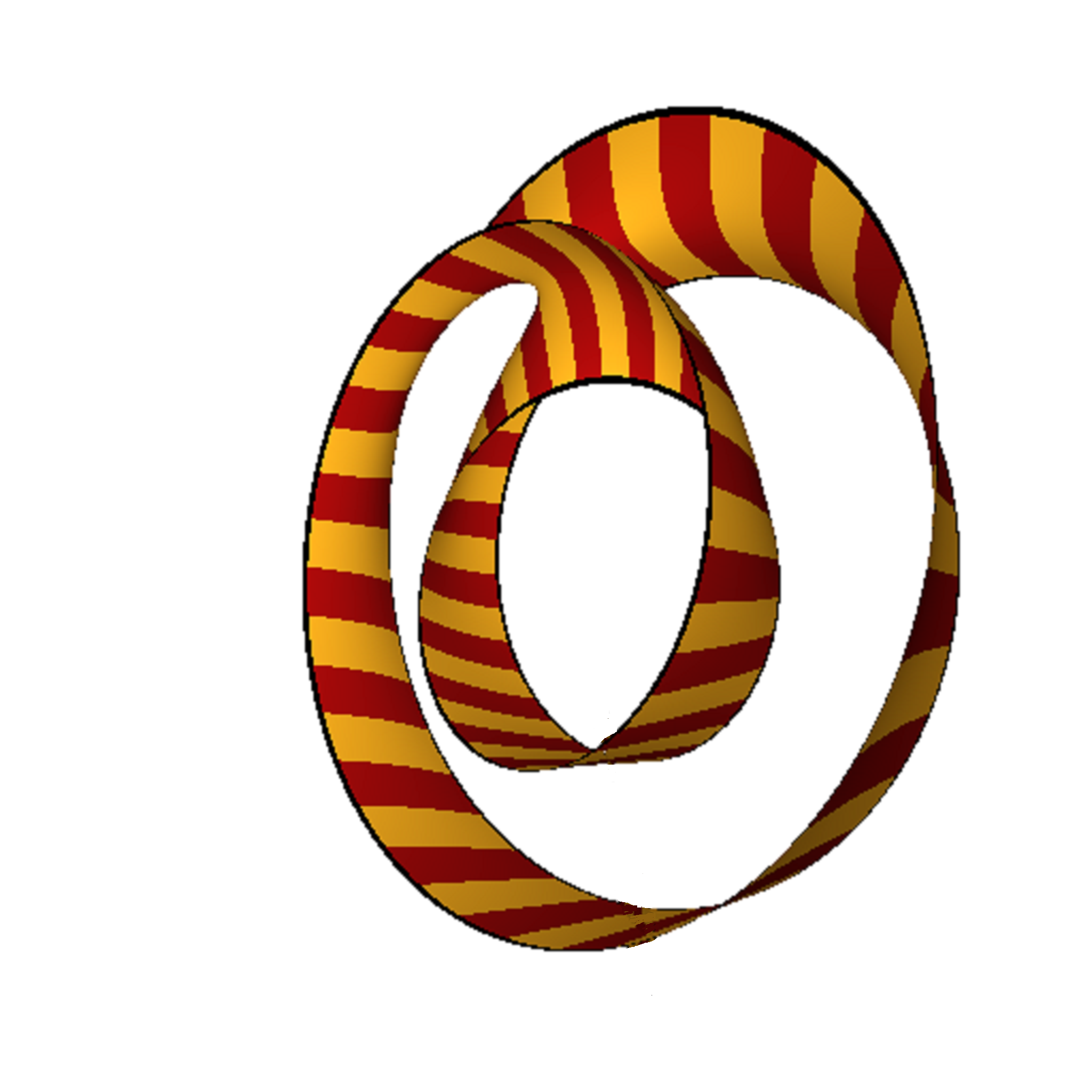} & 
		\includegraphics[width=.3\linewidth]{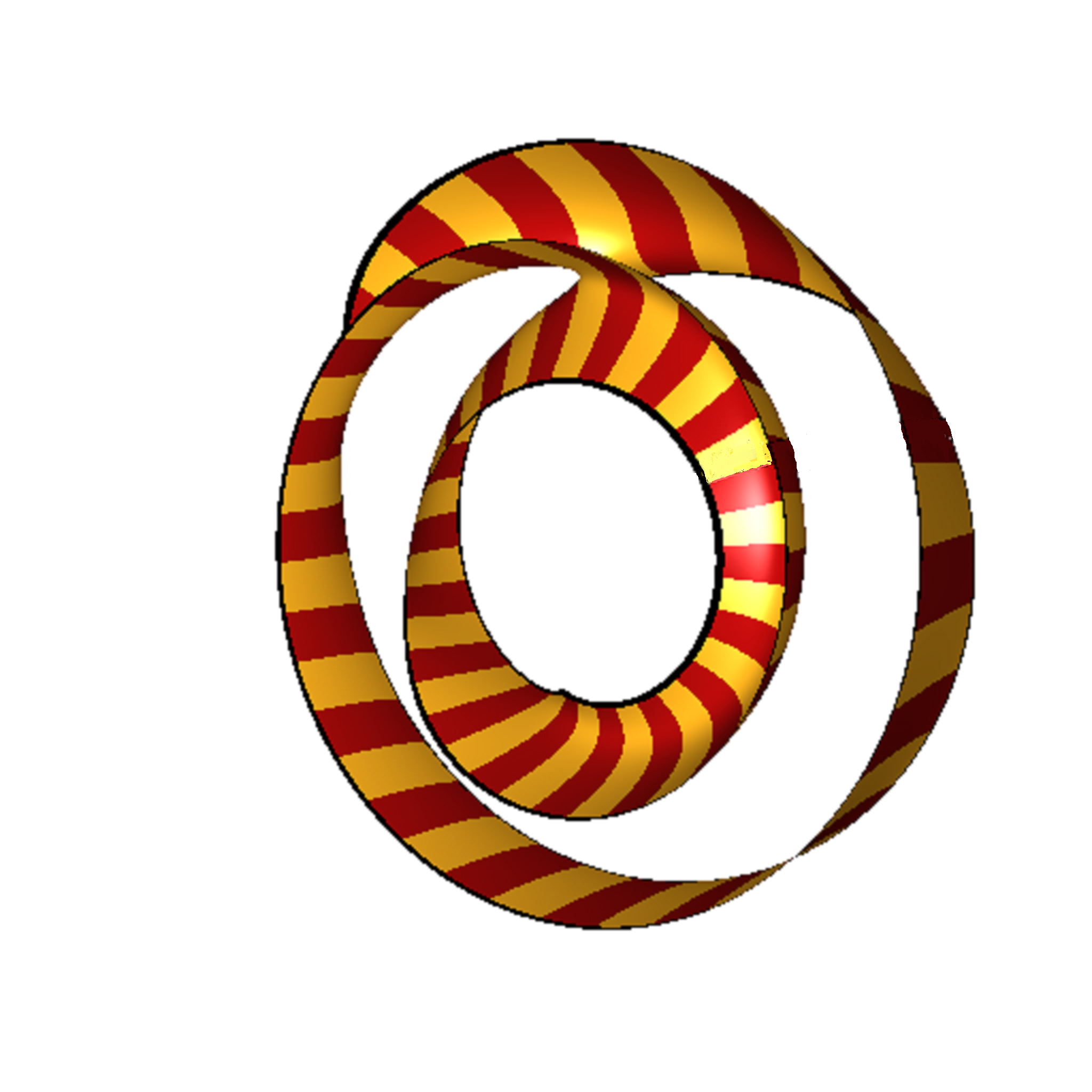} &
		\includegraphics[width=.3\linewidth]{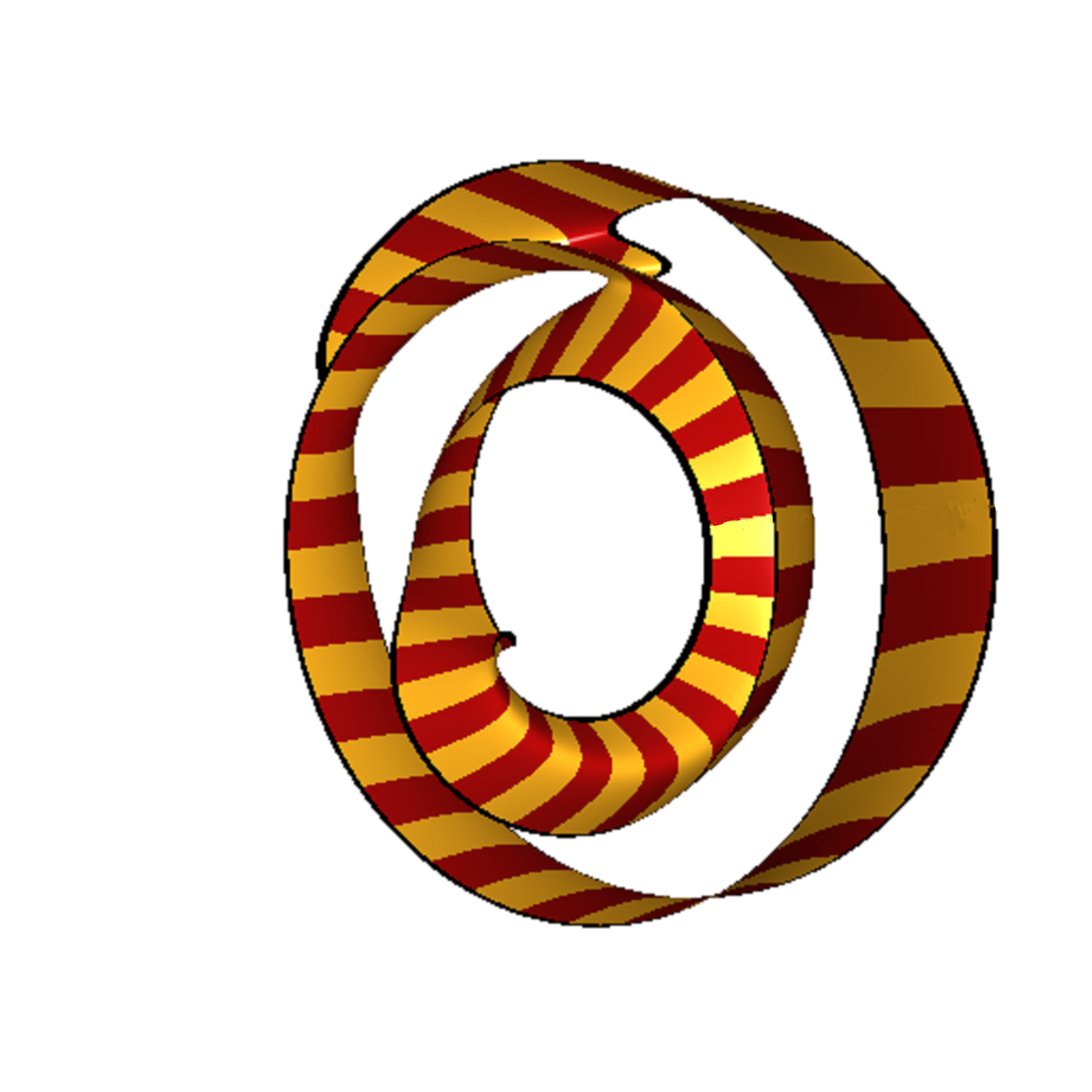}
		\\
		\small $B^{(0)}$ & \small $B^{(0.5)}$ & \small $B^{(1)}=C$
	\end{tabular}
\caption{Deformation of a regular two-point blowup with constant sign distribution}
\end{figure}
\end{exas}

\noindent {\rm \bf{Two Families of Regular Three-point Blowups.}} {\rm Up to now, we have seen examples of families of regular $n$-point blowups for $n =1,2$ and $n=4$ (see Figure 3, Figures 6 and 7 and Figure 4 respectively). We now aim to present two families of regular 3-point blowups.  As above we choose $\rho =2, r = 4, U = \{(x,y) \in \mathbb{R}^2 \mid x^2 + y^2 < 4\}$ for our visualization. }

\begin{exas}\label{Example.Reg.3-Point} {\rm (A) We consider the following example of \cite{Kor} given by:}
	
	$$B := \mathrm{Bl}_U(\underline{f}), \mbox{ {\rm with} } f_0 := \frac{1}{2}({\bf x}-1) + {\bf y}^2 \mbox{ {\rm and} } f_1 := ({\bf x} + \frac{1}{2}){\bf y}.$$
	{\rm We have}
	$$Z = Z_U(\underline{f}) = \{p_1,p_2,p_3\} \mbox{ {\rm with} } p_1 = (-\frac{1}{2},\frac{\sqrt{3}}{2}),  p_2 = (-\frac{1}{2},-\frac{\sqrt{3}}{2}), p_3 = (1,0)$$
	
	\noindent{\rm and hence  $Z$ is the set of vertices of an equilateral triangle centered at the origin $\underline{0} \in \mathbb{R}^2.$ Moreover, it holds} 
	
	$$\mathrm{det}\big(\partial\underline{f}\big)(p_1) = \mathrm{det}\big(\partial\underline{f}\big)(p_2) = -\frac{3}{2} \mbox{ {\rm and} } \mathrm{det}\big(\partial\underline{f}\big)(p_3) = \frac{3}{2}.$$
	
	\noindent{\rm So $B$ is a regular three-point blowup. The sign distribution and hence the embedded isomorphism type of $B$ is given by}
	
	$$\mathrm{sgn}_B(p_i) = \begin{cases} -1, &\mbox{ {\rm for} } i=1,2\\ 1, &\mbox{ {\rm for} } i =3. \end{cases}$$
	
	\noindent{\rm So, in this case we have a \textit{regular three-point blowup with non-constant sign distribution.}}
	
	\noindent{\rm  We consider the family of matrices}
	$$\big(M^{(t)}\big)_{t \in [0,1]} \mbox{ {\rm with} } M^{(t)} := \begin{pmatrix} t{\bf x} + (1-t) & -2t \\ 3t & t{\bf y} + (1-t)\end{pmatrix}.$$ 
	{\rm  As $\mathrm{det}(M^{(t)}) = (7+{\bf x}{\bf y} - {\bf x} - {\bf y})t^2 + ({\bf x} + {\bf y} - 2)t +1,$ and hence $\mathrm{det}(M^{(t)})(p) > 0$ for all $(x,y) = p \in U$ and all $t \in [0,1],$ it follows that}
	$\big(B^{(t)} := \mathrm{Bl}_U(\underline{f}M^{(t)})\big)_{t \in [0,1]}$ % \mbox{ {\rm with} } M^{(t)} = \begin{pmatrix} t{\bf x} + (1-t) & -2t 3t & t{\bf y} + (1-t)\end{pmatrix}$$
	{\rm is an isotopic family of regular three-point blowups which non-constant sign distribution, whose visualization is presented in 
		Figure 8 for $t= 0,  0.33, 0.5,  1.$}
	
	\begin{figure}
		\centering
		\begin{tabular}{cc}
			\includegraphics[width=.3\linewidth]{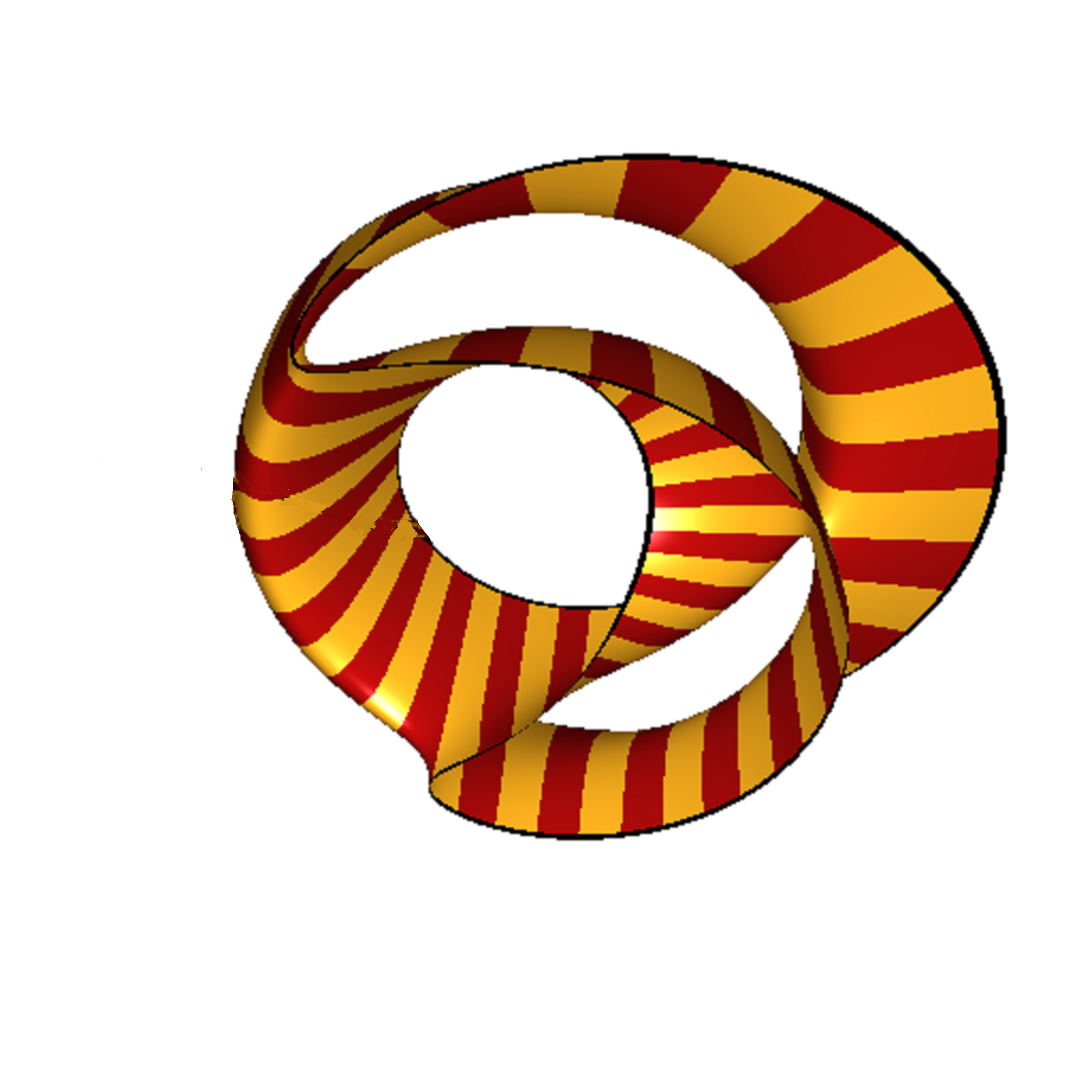} &  \includegraphics[width=.3\linewidth]{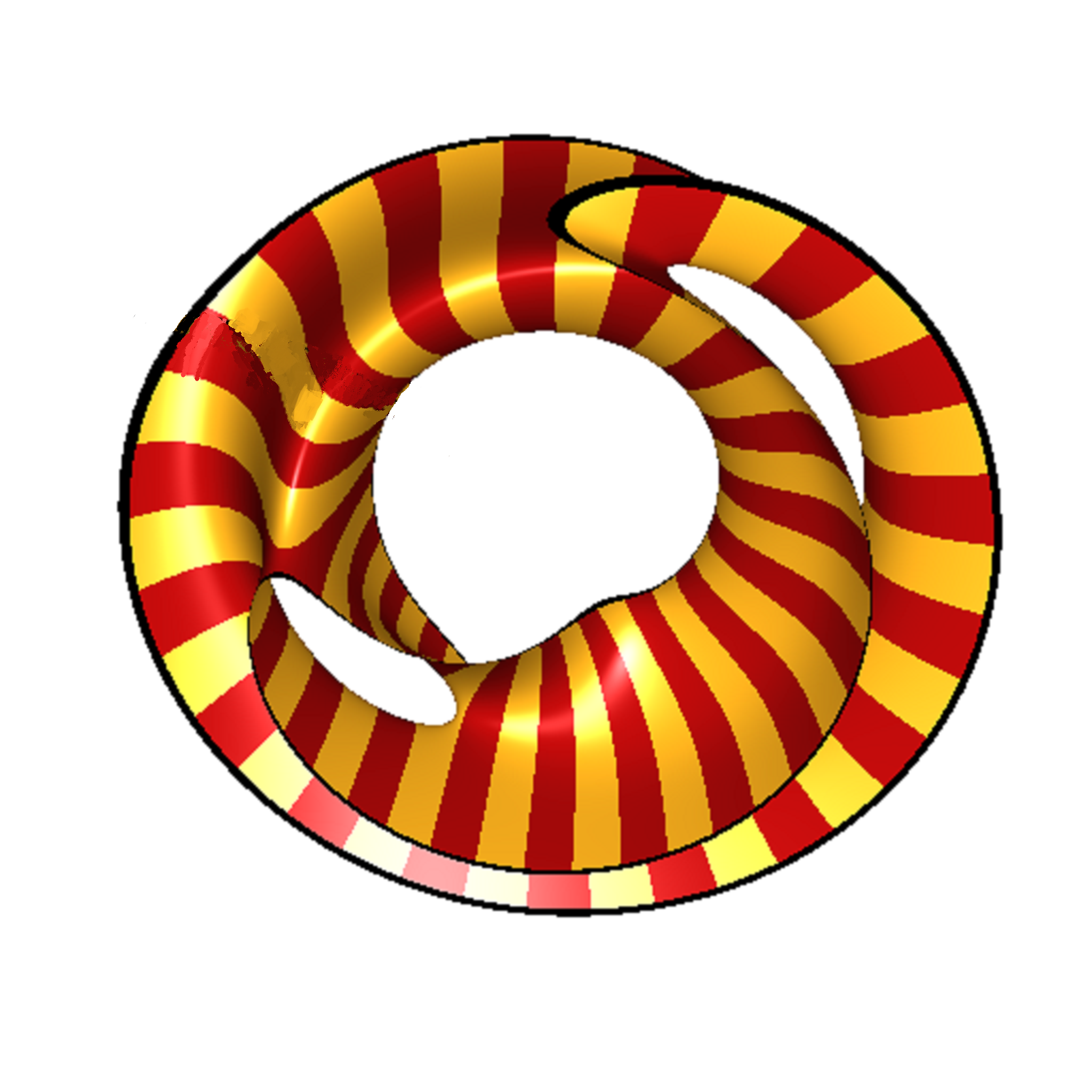} 
			\\ 
			\small $B^{(0)}= B$ &  \small $B^{(0.33)}$  
		\end{tabular} 
		\begin{tabular}{cc}
			\includegraphics[width=.3\linewidth]{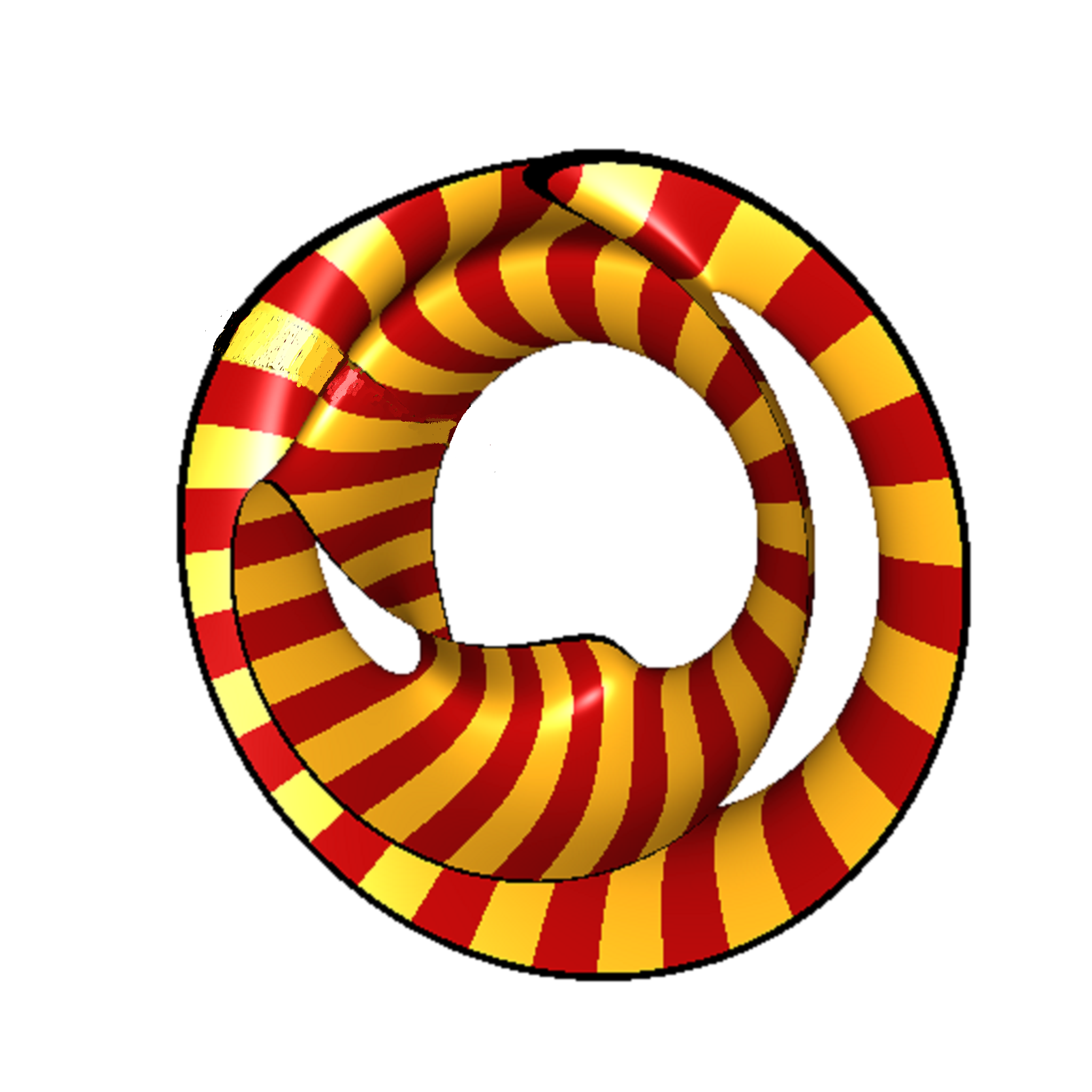} & 
			\includegraphics[width=.3\linewidth]{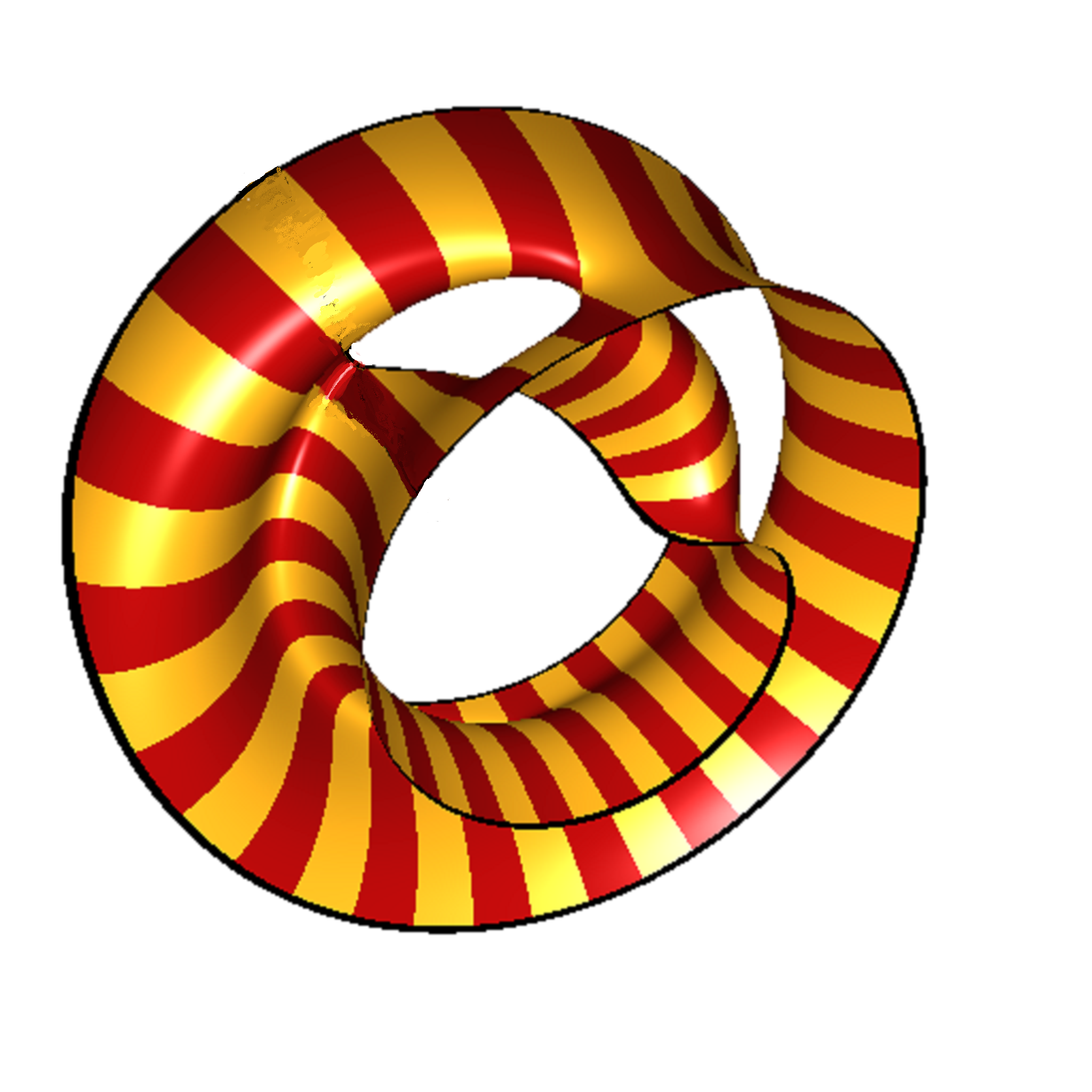} 
			\\
			\small $B^{(0.5)}$ & \small $B^{(1)}$ 
		\end{tabular}
		\caption{Deformation of a regular three-point blowup with non-constant sign distribution}
	\end{figure}

	{\rm (B) We now aim to present a family of regular three-point blowups with constant sign distribution. We choose} 
	$$Z := \{p_1 = (x_1,y_1) = (-\frac{1}{2},\frac{\sqrt{3}}{2}), p_2 = (x_2,y_2) = (-\frac{1}{2},-\frac{\sqrt{3}}{2}), p_3 = (x_3,y_3) = (1,0)\}$$ 
	{\rm as in part (A). Our first aim is to find a strongly regular pair $\underline{f} = (f_0,f_1) \in \mathbb{R}[{\bf x},{\bf y}]^{2}$ with respect to $Z$ on $U$ 
		(see Definition~\ref{Definition EI2}) such that $\mathrm{det}\big(\partial \underline{f}\big)(p_i) = 1$ for $i=1,2,3.$ We do this according to the procedure 
		suggested in the proof of Lemma~\ref{Lemma EI1}, but with the r\^oles of ${\bf x},$ ${\bf y}$ and of $f_0, f_1$ exchanged respectively. We thus set}
	$$f_1 = \prod_{i=1}^3({\bf y} - y_i) = {\bf y}^3 - \frac{3}{4}{\bf y} = {\bf y}({\bf y}^2-\frac{3}{4}) \mbox{ {\rm and} } f_0 = h({\bf y})\big({\bf x} - g({\bf y})\big) 
	\mbox{ {\rm with} }$$ 
	$$\mathrm{deg}(h), \mathrm{deg}(g) \leq 2, \mbox{ {\rm and} } g(y_i) = x_i, \quad h(y_i) = \frac{1}{\prod_{j \neq i}(y_i - y_j)} \mbox{ {\rm for} } i=1,2,3.$$
	{\rm  So}
	$$g({\bf y}) = -2{\bf y}^2 +1 \mbox{ {\rm and} } h({\bf y}) = \frac{4}{3}(2{\bf y}^2 - 1), \mbox{ {\rm thus} }$$
	$$f_0 = \frac{4}{3}(2{\bf y}^2-1)({\bf x} + 2{\bf y}^2 -1) = \frac{4}{3}(4{\bf y}^4 + 2{\bf y}^2{\bf x} - 4{\bf y}^2 - {\bf x}+1).$$
{\rm Now, we have $Z_{\mathbb{R}^2}(\underline{f}) = \{p_1,p_2,p_3\}$ 
			and $\mathrm{det}(\partial \underline{f}) = 4(2{\bf y}^2 -1)({\bf y}^2 - \frac{1}{4}),$ so that 
			$\mathrm{det}(\partial \underline{f})(p_i) = 1$ for $i=1,2,3.$ Therefore $B := \mathrm{Bl}_U(\underline{f})$ 
			is a 
			\textit{regular three-point blowup with constant sign distribution} $\mathrm{sgn}_B(p_i) = 1$ for $i=1,2,3$.}\\
	
	{\rm Our present example illustrates at this point, that the method suggested in the proof of Lemma~\ref{Lemma EI1} tends to furnish pairs of polynomials
		which may be simplified without changing the sign distribution (and hence the isomorphism type) in $\mathfrak{Bl}_U^{\mathrm{reg}}(Z).$ Namely, by setting}
	$$h_0 := \frac{3}{4}f_0 - 4{\bf y}, f_1 =  2{\bf x}{\bf y}^2 - {\bf y}^2 - {\bf x} +1 \mbox{ {\rm and} } h_1 := 4 f_1 = 4{\bf y}^3 - 3{\bf y}$$
	\noindent{\rm we get indeed $Z_{\mathbb{R}^2}(\underline{h}) = \{p_1,p_2,p_3\}$ and $\mathrm{det}(\partial \underline{h}) = 3(2{\bf y}^2-1)(4{\bf y}^2-1)$ so that  
		$\mathrm{det}\big(\partial \underline{h}(p_i)\big) > 0$ and hence $\mathrm{sgn}_{\underline{h}}(p_i) = \mathrm{sgn}_{\underline{f}}(p_i) = 1$ for $i=1,2,3.$}
	
	{\rm For a better visualization of the blowup $\mathrm{Bl}_U(\underline{h})$ we modify it slightly by 
		interchanging the two indeterminates ${\bf x},$ ${\bf y}$ and the two polynomials $h_0$ and $h_1$ 
		(which interchanges the coordinates of the common zeros of the two polynomials, and does not affect their 
		Jacobian determinant -- and hence preserves the (constant) sign distribution of the corresponding blowup), and by multiplying the first of them by 
		$\frac{1}{3}$ (which gives a dilatation of the blowup in the "direction of the fibers"). 
		So, we shall consider the blowup $B = \mathrm{Bl}_U(\underline{g})$ with $\underline{g} = (g_0,g_1),$}
	\[
	g_0 = {\bf x}({\bf x}^2-\frac{3}{4}) \text{ {\rm and} } g_1 = 
	2{\bf x}^2{\bf y} - {\bf x}^2 - {\bf y} +1 		
	\]
	{\rm under the deformation given by the family of matrices $M^{(t)}$ of part (A). This time, for the sake of virtual simplicity, we present with our method of visualization only the single blowup $B^{(t)} = \mathrm{Bl}_U(\underline{g}M^{(t)})$
		for $t = 0.5$ (see Figure 9) and the two affine charts of the blowup $B^{(0)}$ given respectively by $g_1({\bf x},{\bf y}) - {\bf z}g_0({\bf x},{\bf y}) = 0$ and $g_0({\bf x},{\bf y}) - {\bf z}g_1({\bf x},{\bf y}) = 0$ (see Figure 10). The two charts were visualized 
		by means of {\sc Mathematica}.}

	\begin{figure}
		\centering
		
		\includegraphics[width=.35\linewidth]{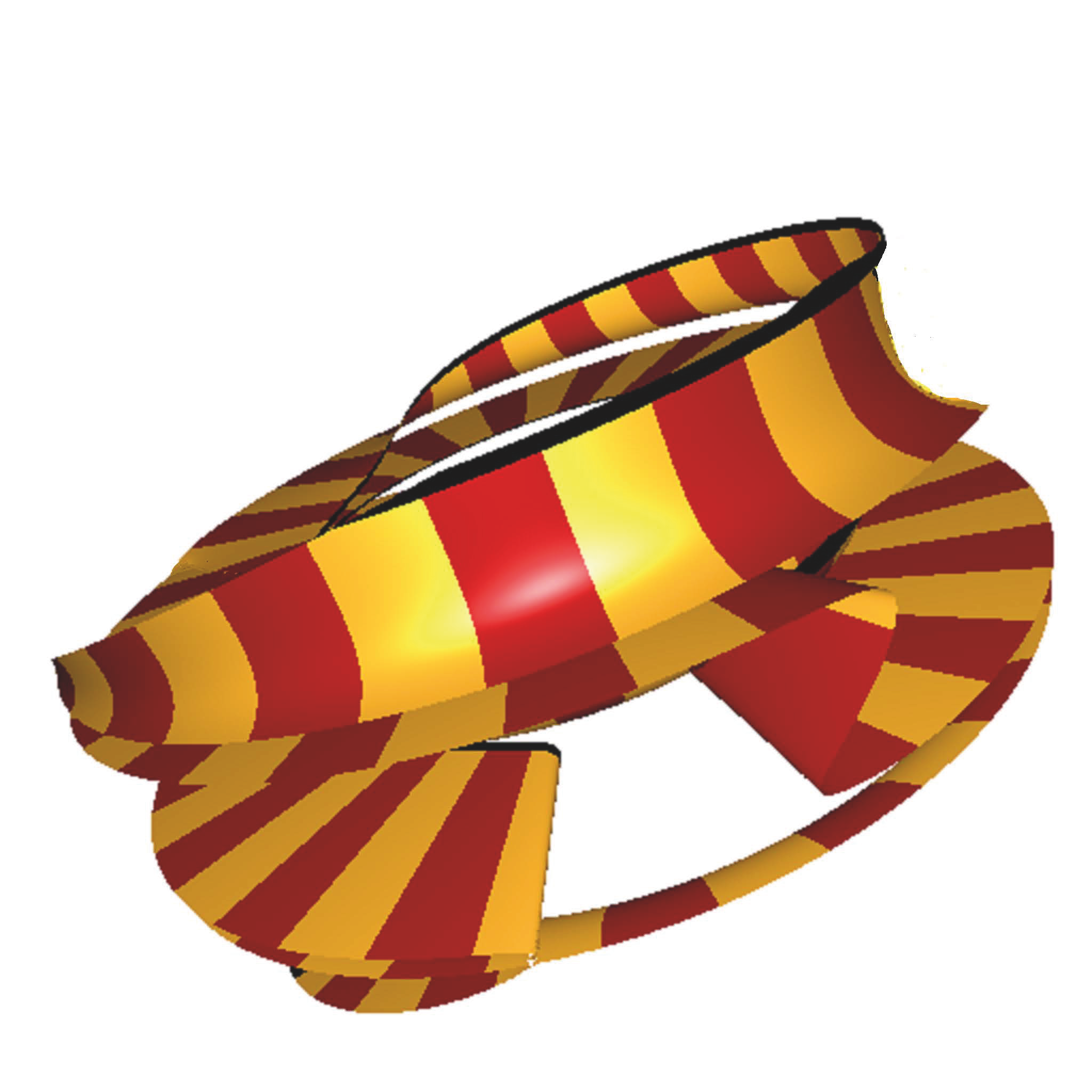} 				 
		
		\caption{Deformation of a regular three-point blowup with constant sign distribution for $t = 0.5$}
		
	\end{figure}
	\begin{figure}
		
		\begin{tabular}{cc}
			\includegraphics[width=.35\linewidth]{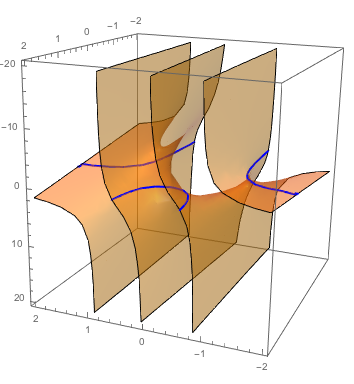} &  \includegraphics[width=.32\linewidth]{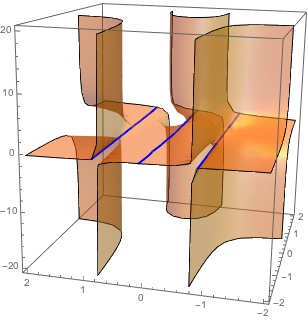} 
			\\ 
			{\rm{1. Chart }} &   {\rm{2. Chart }} 
		\end{tabular} 
		
		\caption{Two charts of a regular three-point blowup with constant sign Distribution} 
		
	\end{figure}
	
\end{exas}

\medskip

%\clearpage

\end{document}